\documentclass[journal,twoside,web]{ieeecolor}
\usepackage{defs}

\definecolor{Orchid}{RGB}{218,112,214}
\definecolor{Bittersweet}{RGB}{254,111,94}
\definecolor{CornflowerBlue}{RGB}{100,149,237}
\definecolor{Goldenrod}{RGB}{218,165,32}
\definecolor{LimeGreen}{RGB}{50,205,50}
\definecolor{Teal}{RGB}{0,128,128}
\definecolor{metabluee}{HTML}{B5CDF5}
\definecolor{metablue}{HTML}{0064E0}
\definecolor{metafg}{HTML}{1C2B33}
\definecolor{metabg}{HTML}{F1F4F7}
\definecolor{metabgdeep}{HTML}{D9EFFF}
\definecolor{metagreen}{HTML}{EAFFE8}
\definecolor{metagreen}{HTML}{FCFFEE}
\definecolor{metared}{HTML}{FFEAE8}

 \usepackage[framemethod=tikz,xcolor=true]{mdframed}

\newmdenv[backgroundcolor=metabg, roundcorner=0pt, skipabove=4pt, linewidth=0pt, innertopmargin=-1pt]{myOCP}

\begin{document}
\pdfminorversion=9 
\pdfcompresslevel=9
\pdfobjcompresslevel=50

\title{Constrained minmax density transportation for linear parabolic PDEs: a numerical optimal control perspective}

\author{Siddhartha Ganguly, Vaibhav Upadhyay, Kenji Kashima, and Debasish Chatterjee
\thanks{This work was not supported by any funding agency.}
\thanks{Siddhartha Ganguly is with the Aerospace Engineering Department, Georegia Institute of Technology, USA (e-mail: \textsf{sganguly41@gatech.edu}).} 
\thanks{Vaibhav Upadhyay is with The Grainger College of Engineering, Electrical \& Computer Engineering, University of Illinois Urbana-Champaign, USA (e-mail: \ \textsf{vu11@illinois.edu}).}
\thanks{Kenji Kashima is The Applied Mathematics and Physics Department, Graduate School of Informatics, Kyoto University, Japan (e-mail: \textsf{kk@i.kyoto-u.ac.jp}).}
\thanks{Debasish Chatterjee is with the Centre for Systems and Control, Indian Institute of Technology Bombay, India (e-mail: \textsf{dchatter@iitb.ac.in}).}
}

\newcommand{\sid}[1]{{\color{red}[Sid: #1]}}

\maketitle

\begin{abstract}
This article introduces a numerical optimal control framework for minmax constrained density control for a class of noisy linear parabolic partial differential equations (PDEs), in particular the noisy heat equation. The goal is to transport an initial density to a target density while minimizing a specified cost with respect to control actions and maximizing it with respect to disturbances, all within a fixed time horizon while satisfying given convex path constraints. To address this, the spatial derivatives in the PDE are discretized using finite-difference approximations, transforming the problem into a system of ordinary differential equations in time. The admissible space of control and disturbance trajectories is then finitely parametrized, and the resulting optimal control problem is formulated as a convex semi-infinite program (SIP) under mild assumptions. By leveraging new numerical tools from convex SIP theory, we establish guarantees for \emph{exact solutions} that account for constraint satisfaction under an infinite family of disturbance realizations, and we establish an optimization-based computationally efficient algorithm to recover these solutions. Comprehensive numerical examples to demonstrate and validate our findings are included.
\end{abstract}


\section{Introduction}\label{sec:intro}

Direct methods offer a variety of optimization-based numerical tools for precise control of dynamical systems while adhering to a prespecified set of constraints. The underlying controlled dynamical system may be governed by a set of Ordinary Differential Equations (ODEs) or by a set of Partial Differential Equations (PDEs). The ultimate goal of a constrained optimal control problem (OCP) is to determine a constrained control trajectory, denoted by \(u\as(\cdot)\), that steers the states of the underlying dynamical system from their initial configuration to a desired state, while adhering to a pre-specified set of constraints along the path. In the case of ODEs with constraints, the primary focus is typically on steering a specific (sometimes, an ensemble of) initial state, represented by a point \(x_{\mathrm{init}} \in \Rbb^d\), to a designated final state, \(x_{\mathrm{target}} \in \Rbb^d\). In contrast, when dealing with PDEs, the objective becomes more intricate, as it requires transporting an entire initial \emph{density} function \(y(\cdot)\), defined over a spatial domain, to a desired \emph{density} function \(y_{\Omega}(\cdot)\). The focus shifts from managing individual state transitions to shaping the evolution of a distributed quantity, ensuring that the transportation adheres to the PDE dynamics and satisfies the boundary condition and convex path constraints. Moreover, since uncertainties (in the model or in the form of exogenous disturbances) in real-world applications are omnipresent, it is important to account for them at the synthesis stage. This article delves into the second case involving PDEs, specifically the \emph{density steering or transportation} problem within PDE-constrained optimal control, \emph{despite a class of uncertainties}, with an emphasis on developing a computationally viable algorithm for the noisy minmax setting.

From a numerical standpoint, the optimal control of PDEs necessitates specialized techniques to handle the inherent infinite-dimensional structure of the problem. Let us outline several broad methodological approaches and perspectives, without attempting an exhaustive survey of the literature.

\begin{enumerate}[label=\textup{(N-\alph*)}, leftmargin=*, widest=b, align=left]

\item \label{numtech:I:DTO} In the \emph{discretize-then-optimize} approach or \emph{direct numerical} methods, the PDE is discretized in space, via finite difference, finite element, or spectral methods, and the PDE-constrained OCP is transformed into an ODE-constrained OCP \cite{ref:ghobadi2009discretize,ref:Hinze2012,ref:DTO:herzog2023} with an uncountable family of constraints. The resulting ODE-constrained OCP is then solved via parametrizing the admissible space of control actions, (see \cite{ref:SG:NR:DC:RB-22,ref:QuITOv2}) and enforcing the constraints at the temporal grid points. These existing approaches offer no guarantees for satisfying the uncountably infinite constraints indexed by time, let alone those arising from the uncountable set of disturbance realizations. We note in advance that the approach established in this work broadly falls within the discretize-then-optimize-type category. Moreover, the presence of disturbances in the PDE complicates the setting and necessitates the development of new numerically viable methods that enable accurate solutions while ensuring the satisfaction of an uncountable family of constraints. 

\item \label{numtech:II:OTD} In the \emph{optimize-then-discretize} or \emph{indirect} approach, optimality conditions for the infinite-dimensional problem are first derived using Pontryagin's Maximum Principle (PMP) \cite[Chapter 3, \S 3.6]{ref:FT:PDE:OptConBook}, \cite{ref:AntKouLacRid-18}. These resulting boundary value problems are then discretized \cite{ref:ito2008lagrange} and solved using indirect multiple shooting methods \cite{ref:carraro2014indirect} through root-finding techniques. While indirect methods based on the PMP are accurate, they are limited to a restricted class of systems and are not applicable in general, constrained, and noisy settings.

\item \label{numtech:IV:MachLearnTech} Another line of research focuses on learning the solution operator of the PDE by parametrizing it using neural network architectures \cite{ref:ML:PDEOperatorLearn,ref:ML:operator:learning:PDE, ref:VerWinRuthBart-24}. This approach involves training models that approximate the solution operator bypassing discretization methods. Notable examples of such methods include Physics-Informed Neural Networks (PINNs), Deep Operator Networks (DeepONETs) \cite{ref:ML:PINN:PDE}, generative models \cite{ref:ML:GenPDECon,ref:ML:DiffPhyCon:PDE}, etc. These methods are largely confined to disturbance-free and unconstrained scenarios.

\item \label{numtech:V:OthTech} Other methods involve \emph{reduced-order modeling} (ROM) \cite[Chapter 3, \S 3.7.2]{ref:FT:PDE:OptConBook} techniques, such as \emph{proper orthogonal decomposition} (POD) or \emph{dynamic mode decomposition} (DMD), which aim to reduce the dimensionality of the PDE system by capturing the dominant modes of the system. For parabolic PDEs, in unconstrained settings, \emph{spectral and variational techniques} leverage basis functions (e.g., Fourier series or wavelets) to represent both the solution and control functions, and solve the ensuing finite-dimensional optimization problems numerically. Moreover, primal-dual active set strategies have gained significant attention in recent years \cite{ref:kunisch2002primal}, particularly in elliptic problems, due to their ability to handle constraints and improve solution accuracy. In \emph{adjoint-based methods}, from the optimality conditions the adjoint equations are derived to compute gradients efficiently with respect to the control variables via numerical projected gradient descent techniques \cite[Chapter 3, \S 3.7]{ref:FT:PDE:OptConBook}, \cite{ref:hinze2008optimization,ref:gruver1981algorithmic}. The existing algorithms derived from these techniques are ineffective in the presence of noise, whereas those capable of handling noise \cite{ref:JDBR-18} are unable to accommodate constraints.
\end{enumerate}

\vspace{1mm}

\textbf{Our contributions:} Against this backdrop, here are our advancements:
\begin{enumerate}[label=\textup{\Alph*.}, leftmargin=*]
\item \textbf{A novel class of problems and approach:} In this article we focus on a class of PDE-constrained OCPs where the underlying PDE is given by a one-dimensional boundary-controlled parabolic PDE, in particular the heat equation, of the form \(\stpde_t(\spacepde,t) = c \stpde_{\spacepde\spacepde}(\spacepde,t)\) with boundary conditions \(\stpde(0, t) = 0\) and \(\stpde(1, t) = \contpde(t) + \uncertpde(t)\). The control input \(v(\cdot)\) is acting at the boundary and \(\uncertpde(\cdot)\) represents the uncertainty term resulting from discrepancies between the designed and actual control inputs in practical implementations, which arises due to non-ideal and/or unmodeled actuator behavior.\footnote{See \S\ref{sec:prob_statement} for more details on the problem data.}

For these classes of PDEs, we explore the problem of \textbf{robust density transportation}, where the initial density \(y(\cdot, 0) \Let \bar{y}(\cdot) \) is to be transported to a desired final density \(y(\cdot, \horizon) \Let y_{\Omega}(\cdot) \), subjected to bounded disturbances and control actions. The ensuing \emph{minmax} problem is infinite dimensional and to tackle this in a numerically viable way, we adopt the well-known semi-discretization approach \cite{ref:SA:AD:BG:error:SciComp00, ref:AD:semi_disc:thesis97}, to convert the PDE to a finite-dimensional linear system of ODEs.

\item \textbf{A novel numerical approach:} For the ODE-constrained minmax problem in continuous time, it is important to note that the constraints must hold for all admissible disturbance realizations, rendering the resulting problem a \emph{semi-infinite} robust programming problem \cite[\S 1, p. 3]{ref:minmax:ness:vinter}. To address this, rather than employing conventional robust optimal control synthesis methods that often rely on conservative approximations, our approach harnesses recent advancements in convex semi-infinite programming reported in \cite{ref:DasAraCheCha-22}. Notably, for the finite-dimensional continuous-time problem, we establish guarantees of \embf{exact solutions} and satisfaction of an uncountable family of constraints. Leveraging our theoretical results, we also provide a numerical architecture to compute the control actions.

\item \textbf{Comparison with existing literature:} We note that the existing approaches discussed in \eqref{numtech:I:DTO}--\eqref{numtech:V:OthTech} cannot be employed to solve the class of problems considered here, not even the minmax ODE-constrained problem, since they are designed for disturbance-free scenarios. Nonetheless, for comparison, we compare our results with the scenario-approach \cite{ref:MC_SG-18}, which is also a technique to solve robust optimization problems. These results are reported in \S\ref{sec:numexp}.

\end{enumerate}

\textbf{Perspectives and applications:}
Our formulation is motivated by distributed-parameter systems in which an entire spatial profile must be shaped reliably in the presence of implementation errors and exogenous perturbations. Such situations arise naturally in thermal regulation \cite{ref:theremal:III}, distributed heating and cooling \cite{ref:Thermal:I,ref:Thermal:II}, material processing \cite{ref:material}, diffusion of chemical concentrations \cite{ref:chemical}, electromagnetic molding machines \cite{ref:electromagnics}, robotic swarms \cite{ref:elamvazhuthi2015optimal}, and other boundary-actuated parabolic systems. In these settings, the control objective is usually to transport a spatial density toward a desired terminal configuration while respecting constraints. Disturbances may enter through actuator mismatch, uncertain boundary fluxes, unmodeled environmental effects, or discrepancies between the commanded and realized inputs. 

A nominal synthesis procedure may therefore produce controls that perform well for the disturbance-free PDE but fail to satisfy the constraints, or fail to achieve the desired terminal profile, under admissible perturbations. The minmax formulation is natural \cite{ref:BSM_IS-04,ref:BM_KZ-94,ref:BSM:AMO} and it addresses this issue at the synthesis stage: the control is chosen against the worst admissible disturbance trajectory, and the constraints are required to hold uniformly over time and over the disturbance class. From this viewpoint, the heat equation serves as a canonical model problem for developing a numerically tractable robust density transportation framework for parabolic PDEs. The resulting approach is relevant for robust thermal management, process control, diffusion-driven manufacturing systems, and more generally for constrained control of distributed-parameter systems where reliability under bounded uncertainty is as important as nominal performance.


\noindent\textbf{Organization:} This article is structured as follows. The core problem addressed in this work is formulated in \S \ref{sec:prob_statement}, where we provide a detailed description of the problem setup and its key challenges. The main contribution of this article is presented in \S\ref{sec:main_result}, where we record our theoretic results and the algorithmic architecture concerning the reformulated semi-infinite programming problem. To demonstrate the applicability of our algorithmic architecture, we provide a couple of numerical examples in \S\ref{sec:numexp}, showcasing its effectiveness and robust performance in the presence of disturbances. The complete proofs of our results are presented in Appendix~\ref{appen:thrm:proofs}.

\section{Problem Statement}\label{sec:prob_statement}
We employ standard notations here. We let \(\N \Let \aset{1,2,\ldots}\) denote the set of positive integers, \(\Nz \Let \N \cup \{0\}\) denote the set of non-negative integers, and \(\Z\) denote the integers. For \(d \in \N\), the vector space \(\Rbb^d\) is assumed to be equipped with standard inner product \(\inprod{v}{v'}\Let v^{\top}v'\) for every \(v,v' \in \Rbb^d\). For \(M\) an arbitrary subset of \(\Rbb^d\), by \(\intr{M}\) we denote the interior of \(M\). Let \(\nu \in \N\); for \(X\) and \(Y\) nonempty open subsets of the real line, the set of \(\nu\)-times continuously differentiable functions from \(X\) to \(Y\) is denoted by \(\mathcal{C}^{\nu}(X;Y)\), and we denote the second-order Sobolev space over the domain \(X\) by \(H^2(X)\); these spaces are equipped with their usual norms \cite[Chapter 1]{ref:adams2003sobolev}.
Let \(I\subset\Rbb\) be a compact interval, let \(m,\nu\in \N\) and let \(\holder\in\lorc{0}{1}\). We denote by \(\mathcal C^{\nu,\holder}(I;\Rbb^m)\) the space of all functions \(f:I\to\Rbb^m\) such that \(f\in \mathcal C^\nu(I;\Rbb^m)\) and \(f^{(\nu)}\) is \(\holder\)-H\"older continuous on \(I\). The H\"{o}lder seminorm of \(f^{(\nu)}\) is defined by
\begin{align}
[f^{(\nu)}(\cdot)]_{\holder}\Let \sup_{\substack{s,t\in I\\ s\neq t}}\frac{\norm{f^{(\nu)}(t)-f^{(\nu)}(s)}}{|t-s|^\holder}.\nn
\end{align}
\(\mathcal{C}^{\nu,\holder}(I;\Rbb^m)\) is equipped with the norm \cite[Chapter 1, \S 1.27]{ref:adams2003sobolev}
\begin{align}\label{eq:holder:norm}
\norm{f(\cdot)}_{\mathcal C^{\nu,\holder}}\Let \sum_{j=0}^{\nu}\norm{f^{(j)}(\cdot)}_\infty+[f^{(\nu)}(\cdot)]_{\holder},
\end{align}
where \(\norm{f^{(j)}(\cdot)}_\infty\) is the usual uniform norm.

\vspace{1mm}

Fix \(\horizon>0\), define \(\spacedomainpde \Let \loro{0}{1}\), and let us consider the following parabolic PDE\footnote{The algorithms established ahead follows for a more general class of PDEs; see Remark \ref{rem:on:general:systems:semi-disc} ahead. We picked to heat equation for simplicity of exposition.}
\begin{equation}
    \label{eq:parabolic_pde}
    \stpde_t(\spacepde,t) = \stpde_{\spacepde\spacepde}(\spacepde,t)\,\,\text{for }(\spacepde,t) \in \domainpde,
\end{equation}
with the boundary conditions
\begin{equation}
    \label{eq:boundary_condition}
    \stpde(0, t) = 0, \text{ and } \stpde(1, t) = \contpde(t) + \uncertpde(t) \,\,\, \text{for } t \in \lcrc{0}{T}.
\end{equation}
We assume that the following problem data are given: 
\begin{enumerate}[label=\textup{(H-\alph*)}, leftmargin=*, widest=b, align=left]

\item \label{eq:prob:data:1} 
The PDE \eqref{eq:parabolic_pde}--\eqref{eq:boundary_condition} evolves in the second-order Sobolev space, i.e., for each fixed \(t \in \lcrc{0}{\horizon}\), \( \stpde(t) \Let \stpde(\cdot, t) \in H^2(\Omega)\). The initial state is given by
\begin{equation}\label{eq:init_final_conditions}
    y(\cdot, 0) \Let \parampde(\cdot) \in \conti{}(\Omega; \Rbb). \nn
\end{equation}

\item \label{eq:prob:data:2} 
To ensure the continuity of \(\stpde(\cdot, 0)\) on \(\ol{\spacedomainpde}\), the initial and boundary data must satisfy the following \emph{compatibility conditions}:
\begin{equation*}
    \parampde(0) = 0, \quad \parampde(1) = \contpde(0) + \uncertpde(0).
\end{equation*}

\item \label{eq:prob:data:3} 
Fix a H\"{o}lder exponent \( \holder \in\lorc{0}{1}\) and let \(\adboundv,\adboundw>0\). 
Let \(\admcontpde \subset \Rbb\) and \(\uncertsetode \subset \Rbb\) be given nonempty and compact intervals containing \(0 \in \Rbb\) in their respective interiors.
The boundary control value at \(t=0\) is prescribed by \(v_0\in \admcontpde\), and \(\bar{y}(1)-v_0\in \uncertsetode\). Let \begin{align*}
\mathcal{R}\Let\mathcal{C}^{3,\holder}(\lcrc{0}{\horizon};\Rbb)\,\,\text{with }\norm{\cdot}_{\mathcal{R}} \text{ as defined in \eqref{eq:holder:norm}} ; 
\end{align*}
the disturbance trajectory \(w(\cdot)\) belongs to
\begin{align*}
\hspace{-3mm}\mathcal{W} \Let  \left\{w(\cdot)\in \mathcal{R}\;\middle\vert\;  
\begin{array}{@{}l@{}}
\uncertpde(0)=\bar{\stpde}(1)-\contpde_0, \, \uncertpde(t)\in \uncertsetode \\ \text{for all } t\in\lcrc{0}{\horizon},\;\norm{\uncertpde}_{\mathcal{R}}\le \adboundw
        \end{array}
        \right\}.
    \end{align*}

\item \label{eq:prob:data:4} The control trajectory \(v(\cdot)\) belongs to
\begin{align*}    
\hspace{-3mm}\mathcal{V} \Let   \left\{ \contpde(\cdot) \in \mathcal{R} \;\middle\vert\;  
\begin{array}{@{}l@{}}
\contpde(0)=\contpde_0,\; \norm{\contpde}_{\mathcal{R}}\le \adboundv, \contpde(t) + \\ \uncertpde(t)\in \admcontpde  \text{ for all }t\in \lcrc{0}{\horizon} \\ \text{ and for all }w(\cdot)\in\mathcal{W}
        \end{array}
        \right\}.
    \end{align*}
\end{enumerate}
These trajectory classes in \ref{eq:prob:data:3}--\ref{eq:prob:data:4} are regarded as subsets of \(\mathcal{C}^3(\lcrc{0}{\horizon};\Rbb)\) equipped with the standard \(\mathcal{C}^3\) uniform topology and we assume that \(\mathcal{W}\) and \(\mathcal{V}\) are nonempty.

 Let \(\terminalstpde(\cdot) \in \mathcal{C}(\overline{\Omega};\Rbb)\) be the desired final state at time \(\horizon\). Consider the objective function
\begin{align}\label{eq:PDE:objective}
&\objective\bigl(\parampde(\cdot), \contpde(\cdot),\uncertpde(\cdot)\bigr) \Let \frac{1}{2} \int_{\Omega} |\stpde(x,T) - \terminalstpde(x)|^2 \odif{x} \nn \\& + \frac{\lambda}{2}\int_{\tinit}^{\horizon} \contpde(t)^{2} \odif{t} - \frac{\gamma}{2} \int_{\tinit}^{\horizon} \uncertpde(t)^2 \odif{t},
\end{align}
where \(\lambda,\gamma>0\) is are parameters related to the control and the uncertainty penalties. Over the preceding admissible controls and uncertainty realizations along with the problem data \ref{eq:prob:data:1}--\ref{eq:prob:data:4} and the objective function \eqref{eq:PDE:objective}, we consider the finite horizon robust OCP
\begin{myOCP}
\begin{equation}\label{eq:og_OCP}
\begin{aligned}
& \hspace{-5mm}\inf_{\contpde(\cdot) \in \mathcal{V} }	\sup_{\uncertpde(\cdot) \in \mathcal{W} }&& \hspace{-3mm}\objective\bigl(\parampde(\cdot), \contpde(\cdot),\uncertpde(\cdot)\bigr) \\
&  \sbjto		&&  \hspace{-8mm}\begin{cases}
    \text{PDE }\eqref{eq:parabolic_pde},\, \stpde(\spacepde, \tinit)= \bar{\stpde}(\spacepde)\,\, \text{for all } \spacepde \in \spacedomainpde,\\
    \text{the conditions }\eqref{eq:boundary_condition},\,\ref{eq:prob:data:3}\mbox{--}\ref{eq:prob:data:4}.
\end{cases}
\end{aligned}
\end{equation}
\end{myOCP}
%


\begin{myOCP}
 \begin{prop}[Existence of an optimizer for \eqref{eq:og_OCP}]\label{prop:pde:existence}
\blue{Consider the OCP \eqref{eq:og_OCP} along with its data \ref{eq:prob:data:1}--\ref{eq:prob:data:4}. Then \eqref{eq:og_OCP} admits an optimizer: there exists \(\contpde^\star\in\mathcal{V}\) such that
\begin{align*}
    v^\star \in \argmin_{\contpde \in \mathcal{V}}\sup_{\uncertpde \in \mathcal{W}}\objective \bigl(\parampde(\cdot),\contpde(\cdot),\uncertpde(\cdot)\bigr).
\end{align*}
}
\end{prop}   
\end{myOCP}

\begin{proof}
\blue{
Our proof proceeds via the direct method of calculus of variarions \cite{ref:santambrogio2023course}. To this end, we first prove compactness of the admissible sets \(\mathcal{W}\) and \(\mathcal{V}\). Recall that \(\mathcal{R}\Let\mathcal{C}^{3,\holder}(\lcrc{0}{\horizon};\Rbb)\) where \(\holder\in\lorc{0}{1}\), and that \(\mathcal W\) and \(\mathcal V\) are regarded as subsets of \(\mathcal C^3(\lcrc{0}{\horizon};\Rbb)\) equipped with the standard \(\mathcal C^3\)-topology.

Since \(\mathcal{W}\) is bounded in \(\mathcal{C}^{3,\holder}(\lcrc{0}{\horizon};\Rbb)\), the functions in \(\mathcal{W}\), together with their derivatives up to order three, are uniformly bounded, and the third derivatives are uniformly \(\holder\)-H\"older continuous. By the Arzel\`a--Ascoli theorem \cite[Theorem 7.25]{ref:Rud-Analysis} applied successively to the derivatives, \(\mathcal{W}\) is relatively compact \cite[Chapter 1, \S 1.11.1]{ref:EZ:Appl:Func:Anal} in \(\mathcal{C}^3(\lcrc{0}{\horizon};\Rbb)\). Moreover, the conditions
\begin{align*}
 w(0)=\bar{\stpde}(1)-\contpde_0,\, w(t)\in\uncertsetode\text{ for all }t\in\lcrc{0}{\horizon},\, \norm{w}_{\mathcal R}\le \adboundw   
\end{align*}
are closed under convergence in \(\mathcal C^3(\lcrc{0}{\horizon};\Rbb)\). Hence \(\mathcal W\) is compact in \(\mathcal C^3(\lcrc{0}{\horizon};\Rbb)\).

The same argument applies to \(\mathcal{V}\). Indeed, \(\mathcal{V}\) is bounded in \(\mathcal{C}^{3,\holder}(\lcrc{0}{\horizon};\Rbb)\), and the conditions
\begin{align*}
v(0)=\contpde_0,\,\norm{v}_{\mathcal{R}}\le \adboundv,\, v(t)+w(t)\in\admcontpde
\end{align*}
for all \((t,w)\in\lcrc{0}{\horizon} \times \mathcal{W}\), are closed under convergence in \(\mathcal{C}^3(\lcrc{0}{\horizon};\Rbb)\). Since \(\admcontpde\) is closed, if \(v_n\to v\) in \(\mathcal{C}^3(\lcrc{0}{\horizon};\Rbb)\) and \(v_n(t)+w(t)\in\admcontpde\) for every \(t\in\lcrc{0}{\horizon}\) and every \(w\in\mathcal{W}\), then \(v(t)+w(t)\in\admcontpde\) for every \(t\in\lcrc{0}{\horizon}\) and every \(w\in\mathcal{W}\). Therefore \(\mathcal{V}\) is compact in \(\mathcal{C}^3(\lcrc{0}{\horizon};\Rbb)\).

Next we verify continuity of the cost functional. For \(v\in\mathcal V\) and \(w\in\mathcal W\), define the boundary input
\begin{align*}
    \lcrc{0}{\horizon} \ni t \mapsto q(t)\Let v(t)+w(t).
\end{align*}
Let \(\stpde^q\) denote the solution of the heat equation \eqref{eq:parabolic_pde}--\eqref{eq:boundary_condition} with boundary input \(q\). We claim that the mapping \(q \mapsto \stpde^q(\cdot,\horizon)\) is continuous from \(\mathcal{C}^3(\lcrc{0}{\horizon};\Rbb)\) into \(L^2(\spacedomainpde)\). To see this, use the convert the PDE system \eqref{eq:parabolic_pde}--\eqref{eq:boundary_condition} into one with homogeneous boundary values. Define a function \((t,x) \mapsto L^q(x,t) \Let x q(t)\); note that \(L^q(0,t) = 0\) and \(L^q(1,t) = q(t)\). Consider the unknown function 
\begin{align*}
    r^q(x,t) \Let y^q(x,t)-L^q(x,t) = y^q(x,t)-xq(t),
\end{align*}
for a all \((x,t) \in \Omega \times \lcrc{0}{\horizon}\). Then \(r_q(\cdot)\) satisfies homogeneous Dirichlet boundary conditions \(r^q(0,t)=0\) and \(r^q(1,t)=0\) and the PDE in \(r^q(\cdot)\) variable is
\begin{align}
r^q_t=r^q_{\spacepde\spacepde}-\spacepde q'(t)\quad\text{with } r_q(\spacepde,0)=\bar{\stpde}(\spacepde)-\spacepde q(0).
\end{align}
Let \(q_n\to q\) in \(\mathcal{C}^3(\lcrc{0}{\horizon};\Rbb)\). Setting \(e_n\Let r^{q_n}-r^q\), we have \(e_n(0,t)=e_n(1,t)=0\) and
\begin{align}
(e_n)_t=(e_n)_{\spacepde\spacepde}-\spacepde(q_n'(t)-q'(t)). 
\end{align}
with \(e_n(\spacepde,0)=-\spacepde(q_n(0)-q(0))\). 
Let \(S(t)\) be the usual operator semigroup, i.e., the solution operator for the zero-boundary heat equation on \(L^2(\spacedomainpde)\) \cite{ref:SC:VN:motion:TAC25}. Then
\begin{align*}
e_n(\cdot,\horizon)=S(\horizon)e_n(\cdot,0)-\int_0^\horizon \hspace{-2mm}S(\horizon-s)\bigl[\spacepde(q_n'(s)-q'(s))\bigr]\,\dd s.
\end{align*}
Since \(S(t)\) is bounded on \(L^2(\spacedomainpde)\), there is a constant \(C_\horizon>0\) such that \cite[Lemma 3.1]{ref:SC:VN:motion:TAC25}
\begin{align}\label{eq:error:ineq}
 \norm{e_n(\cdot,\horizon)}_{L^2(\spacedomainpde)} & \le C_\horizon\norm{\spacepde}_{L^2(\spacedomainpde)}|q_n(0)-q(0)| \nn \\& + C_\horizon\horizon\norm{\spacepde}_{L^2(\spacedomainpde)}\norm{q_n'-q'}_\infty.  
\end{align}
Therefore \(e_n(\cdot,\horizon)\to 0\) in \(L^2(\spacedomainpde)\). Since
\begin{align}\label{eq:y:express}
\stpde^{q_n}(\cdot,\horizon)-\stpde^q(\cdot,\horizon)=e_n(\cdot,\horizon)+\spacepde(q_n(\horizon)-q(\horizon)),    
\end{align}
form \eqref{eq:error:ineq} and \eqref{eq:y:express}, we obtain
\begin{align*}
\norm{\stpde_{q_n}(\cdot,\horizon)-\stpde_q(\cdot,\horizon)}_{L^2(\spacedomainpde)}\to 0.
\end{align*}
Thus the terminal solution map is continuous.

Now define \((v,w) \mapsto F(v,w)\Let f_\circ(\bar{\stpde}(\cdot),v(\cdot),w(\cdot)).\) Using the preceding continuity of \(q\mapsto \stpde^q(\cdot,\horizon)\), the inclusion \(y_\Omega(\cdot)\in \mathcal{C}(\ol{\spacedomainpde};\Rbb)\subset L^2(\spacedomainpde)\), and the continuity of the maps
\begin{align*}
v\mapsto\int_0^\horizon v(t)^2\,\dd t,\qquad w\mapsto\int_0^\horizon w(t)^2\,\dd t 
\end{align*}
it follows that \(F(\cdot)\) is continuous on \(\mathcal{V}\times\mathcal{W}\). Since \(\mathcal{V}\times\mathcal{W}\) is compact, \(F(\cdot)\) is uniformly continuous on \(\mathcal{V}\times\mathcal{W}\). For each \(v\in\mathcal{V}\), define
\begin{align*}
    \mathcal{V} \ni v \mapsto G(v)\Let\sup_{w\in\mathcal W}F(v,w).
\end{align*}
Since \(\mathcal{W}\) is compact and \(F(v,\cdot)\) is continuous, the supremum is attained for every \(v\in\mathcal{V}\) \cite[Chapter 1]{ref:santambrogio2023course}. We now show that \(G\) is continuous. If \(v_n\to v\) in \(\mathcal{C}^3(\lcrc{0}{\horizon};\Rbb)\), then
\begin{align*}
\abs{G(v_n)-G(v)}\le\sup_{w\in\mathcal W}\abs{F(v_n,w)-F(v,w)}.    
\end{align*}
By uniform continuity of \(F(\cdot)\) on \(\mathcal{V}\times\mathcal{W}\), the right-hand side tends to zero. Hence \(G(v_n)\to G(v)\), and \(G(\cdot)\) is continuous on \(\mathcal{V}\). Finally, existence of a \(v^\star\) follows immediately from the compactness of \(\mathcal V\) and continuity of \(G(\cdot)\). Our assertion stands established.
}
\end{proof}
Observe that \eqref{eq:og_OCP} is a minmax problem over all functions \(v(\cdot) \in \mathcal{V}\) that minimize the objective once all the adversaries \(w(\cdot) \in \uncertset\) have maximized the objective. Thus, in its current form, \eqref{eq:og_OCP} is an infinite-dimensional optimization problem and is numerically intractable in general. To ensure numerical viability, first, we employ a semi-discretization of the PDE \eqref{eq:parabolic_pde}, leading to an ODE in time; this is a standard approach, see \cite{ref:FT:PDE:OptConBook}, \cite[Chapter 4]{ref:RobConPDEbook-2001}, \cite{ref:HasYosTos-12}, \cite{ref:HerKun-10}. Subsequently, we employ a parametrization for the admissible space \(\mathcal{V}\) of control and disturbance \(\uncertset\) using a dictionary of \(\conti{3}(\lcrc{0}{\horizon}; \Rbb)\) functions to induce numerical viability of the corresponding problem.

\begin{rem}[On the choice of \(\mathcal{V}\) and \(\uncertset\)]
In the problem data \ref{eq:prob:data:1}--\ref{eq:prob:data:4}, we picked smooth control and disturbance trajectories to simplify the exposition. This choice also enables the employment of well-established convergence results from standard semi-discretization techniques. From a practical and application-oriented perspective, smooth controls are desirable in many settings --- for instance, in motion control problems governed by PDEs~\cite{ref:SC:VN:motion:TAC25, ref:motion:planning:Rouchon:I, ref:motion:planning:Rouchon:II}. 
\end{rem}

\begin{rem}[Optimal heat source problems]
A practically relevant scenario occurs when \(\contpde(\cdot)\) serves as a heat source distributed across the domain \(\spacedomainpde\). These problems, referred to as \emph{optimal heat source} problems, naturally emerge in applications such as heating metals through electromagnetic induction or microwave energy. In such cases, letting \(s \Let (x,t)\), one can consider an objective function
\begin{align}
&\objective\bigl(\parampde(\cdot), \contpde(\cdot),\uncertpde(\cdot)\bigr) \Let \frac{1}{2} \int_{\Sigma} |\stpde(s) - y_{\Sigma}(s)|^2 \odif{s} \nn \\& + \frac{\lambda}{2}\int_{\Gamma} |v(s)|^2 \odif{s} - \frac{\gamma}{2} \int_{\Gamma} \uncertpde(s)^2 \odif{s}, \nn
\end{align}
where \(\Sigma \Let \mathrm{bd}(\Omega) \times \lcrc{0}{\horizon}\), and \(\Gamma \Let \domainpde\). Of course, a control term will appear in the PDE dynamics \eqref{eq:parabolic_pde} and the boundary conditions must be altered accordingly; see \cite[Chapter 3, \S 3.5.2]{ref:FT:PDE:OptConBook}. This formulation emphasizes minimizing the energy expenditure associated with the control \(v(x,t)\) over the spatial and temporal domains. For a detailed discussion on alternative formulations and theoretical underpinnings, we refer to \cite[Chapter 3]{ref:FT:PDE:OptConBook}. We highlight that the technique we describe is also applicable to the optimal heat source problems, provided the same set of assumptions on the problem data hold.
\end{rem}

\subsection{Semi-discretization of the PDE \eqref{eq:parabolic_pde}}
We take the first step towards translating \eqref{eq:og_OCP} into a numerically viable problem. To this end, we record the semi-discretization procedure in brief. We fix \(n \in \Nz\), and let \(h\Let \tfrac{1}{n+1}\) be a discretization parameter. For the interior points of the spatial domain, we employ \textcolor{black}{a \(\intdisc\)-order} finite difference scheme, \textcolor{black}{while at the boundary}, a \(\extdisc\)-order finite difference scheme is used to discretize the spatial derivatives \cite{ref:JCS:finite_diff:siam04}. Then the semi-discretized dynamics is%
\begin{equation} \label{eq:semi-disc}
 \dot \st(t) = A_n \st(t) + B_n (\contpde(t) + \uncertpde(t))\quad\text{for all } t\in \lcrc{0}{\horizon},
\end{equation}
where 
\begin{itemize}[leftmargin=*]
    \item \(t \mapsto\st(t) \Let \bigl(\stpde(h, t),\stpde(2h, t),\cdots,\stpde(nh,t)\bigr)^{\top} \in \Rbb^n\) is the state variable,
    \item \(\param = \bigl(\parampde(h),\parampde(2h),\cdots,\parampde(nh)\bigr)^{\top} \in \Rbb^n\) is the initial state corresponding to the initial distribution \(\parampde(\cdot)\), and
    \item \(\st_{\spacedomainpde} \Let  \bigl(\stpde_{\spacedomainpde}(h),\stpde_{\spacedomainpde}(2h),\cdots, \stpde_{\spacedomainpde}(nh)\bigr)^{\top} \in \Rbb^n\) is the desired state at the final time.
\end{itemize} 
The system matrix \(A_n \in \mathbb{R}^{n \times n}\) and the actuation matrix \(B_n \in \mathbb{R}^{n \times 1}\) depend on \(\intdisc\) and \(\extdisc\). The next theorem establishes error bounds for the accuracy of semi-discretization employing standard results from \cite[p. 86]{ref:SA:AD:BG:error:SciComp00}.
\begin{prop}{\textcolor{black}{\cite[p. 86]{ref:SA:AD:BG:error:SciComp00}}}
\label{thrm:err_estimates}
\blue{Consider the PDE \eqref{eq:parabolic_pde} with boundary conditions \eqref{eq:boundary_condition} and associated data \ref{eq:prob:data:1}–\ref{eq:prob:data:4}. Fix \(n \in \Nz\) and let \eqref{eq:semi-disc} be the semi-discretized ODE corresponding to the PDE-system \eqref{eq:parabolic_pde}-\eqref{eq:boundary_condition}. For \(\intdisc, \extdisc \in \Nz\) with \(\intdisc \geq 2\) and \(\intdisc \geq \extdisc\) we employ a symmetric \(\intdisc\)-order finite difference scheme at the interior points of the spatial domain, while near the boundaries, a non-symmetric \(\extdisc\)-order finite difference scheme is used to discretize the spatial derivatives. Let \(\stpde(\cdot, \cdot)\) be the solution of the PDE \eqref{eq:parabolic_pde} with boundary conditions \eqref{eq:boundary_condition}. Let \(h = \frac{1}{n+1}\) and for any continuous function \(f(\cdot)\) we let \(R_n f = (f(h),\ldots,f(nh))^{\top}\), and define the error trajectory by
   \(\lcrc{0}{\horizon} \ni t \mapsto \varepsilon(t) = \st(t) - R_n\stpde(\cdot, t) \in \Rbb^n.\)
Then:
\begin{enumerate}[label=\textup{(\ref{thrm:err_estimates}-\alph*)}, leftmargin=*, widest=b, align=left]
    \item \label{thrm:err_estimates_0} For \(\extdisc = 0\) or \(\extdisc = 1\), the error satisfies
    \begin{equation*}
        \norm{\varepsilon(t)}^2 \leq \gamma_1 \left(1 + \mathcal{O}(h) \right) h^{2\intdisc} \text{ for all } t \in \lcrc{0}{\horizon},
    \end{equation*}
    
    \item \label{thrm:err_estimates_1} for \(\extdisc \geq 2\), the error satisfies
    \begin{equation*}
       \hspace{-11mm} \norm{\varepsilon(t)}^2 \leq \gamma_2 \left(1 + \mathcal{O}(h^{2\extdisc - 3}) \right) h^{2\intdisc + 3 - 2\extdisc} \text{ for all } t \in \lcrc{0}{\horizon},
    \end{equation*}
\end{enumerate}
where \(\gamma_1, \gamma_2, c_2 > 0\) are constants, to wit, the rate of convergence is \(h^{\intdisc}\) for \(\extdisc = 0,1\), and is \(h^{\intdisc - \extdisc + 3/2}\) for \(\extdisc \geq 2\).}
\end{prop}
\blue{The proof of Proposition \ref{thrm:err_estimates} is a trivial extension of the results given in \cite[p. 86]{ref:SA:AD:BG:error:SciComp00}, and is omitted.}

We employ the Euler discretization for discretizing the integral in the cost function of the OCP \eqref{eq:og_OCP}. The first term of \(\objective\bigl(\parampde(\cdot),\contpde(\cdot),\uncertpde(\cdot)\bigr)\) in \eqref{eq:og_OCP} is given by:
\begin{align}
    \label{eq:discrete_integral}
    \int_{\spacedomainpde} |\stpde(x,T) - \stpde_{\Omega}(x)|^2 \odif{x} &\approx h\sum_{k=1}^{n} |\stpde(kh,T) - \stpde_{\Omega}(kh)|^2 \nn \\& = h \norm{\st(T) - \st_{\spacedomainpde}}^2.
\end{align}
Thus, the discretized version of the objective function \(\objective\bigl(\parampde(\cdot),\contpde(\cdot),\uncertpde(\cdot)\bigr)\) is
\begin{align}
    \label{eq:discrete_cost}
&\objectivedisc\bigl(\param, \contpde(\cdot), \uncertpde(\cdot)\bigr) = \frac{h}{2} \norm{\st(T) - \st_{\spacedomainpde}}^2  \nn \\& + \frac{\lambda}{2}\int_{\tinit}^{\horizon} \contpde(t)^{2} \odif{t} - \frac{\gamma}{2}\int_{\tinit}^{\horizon} \uncertpde(t)^2  \odif{t}.
\end{align}

\subsection{Parametrization of the admissible space of control and disturbance}
Even after the semi-discretization, the continuous-time robust OCP \eqref{eq:og_OCP} remains numerically intractable in general, so we introduce a parametrization of the admissible control and disturbance trajectories to facilitate viability \cite[\S 1, p.3]{ref:minmax:ness:vinter}.
\begin{defn}\label{defn:discrete_admcon}
\blue{Let \(\aset[]{\dicC_i(\cdot) \suchthat i \in \Nz} \subset  \mathcal{C}^{3,\holder}(\lcrc{0}{\horizon};\Rbb)\) and \(\aset[]{\dicD_i(\cdot)\suchthat i \in \Nz} \subset  \mathcal{C}^{3,\holder}(\lcrc{0}{\horizon};\Rbb)\) be dictionaries of bounded and linearly independent functions.} Fix \(N_{\contpde}, N_{\uncertpde}\in \N\) and define the set of admissible controls and disturbance trajectories \(\dict_{\contpde}\) and  \(\dict_{\uncertpde}\) by 
 \(\dict_{\contpde} \Let \linspan \aset[\big]{ \dicC_i:[0,\horizon] \ra \Rbb \suchthat i \in \aset[]{1, \ldots, N_{\contpde}}} \text{ and } \dict_{\uncertpde} \Let \linspan \aset[\big]{ \dicD_j:[0,\horizon] \ra \Rbb \suchthat j \in \aset[]{1, \ldots, N_{\uncertpde}}}\).
\end{defn}
Based on Definition \ref{defn:discrete_admcon}, we let \(\Reg(t) \Let \bigl(\dicC_1(t)\; \dicC_2(t)\;\ldots \;\dicC_{N_{\contpde}}(t) \bigr) \in \Rbb^{N_{\contpde}}\) and \(\RegD(t) \Let \bigl(\dicD_1(t)\; \dicD_2(t)\;\ldots \;\dicD_{N_{\uncertpde}}(t) \bigr) \in \Rbb^{N_{\uncertpde}}\) for all \(t\in \lcrc{0}{\horizon}\), and let
\begin{align} 
    \label{e:cont_param}
    [0,\horizon] \ni t \mapsto \contparam(t) \textcolor{black}{\Let} \sum_{j=1}^{N_{\contpde}}\Param_{i} \dicC_i(t) \textcolor{black}{=} \inprod{\Param}{\Reg(t)},
\end{align}
where \textcolor{black}{\(\Param \Let (\Param_1, \ldots, \Param_{N_{\contpde}})\in \Rbb^{N_{\contpde}}\)} is the \emph{control coefficient} vector (to be determined) and let
\begin{align} 
    \label{e:uncert_param}
    [0,\horizon] \ni t \mapsto \uncertparam(t) \textcolor{black}{\Let} \sum_{j=1}^{N_{\uncertpde}}\Paramw_{j} \dicD_j(t) \textcolor{black}{=} \inprod{\Paramw}{\RegD(t)},
\end{align}
where \textcolor{black}{\(\Paramw \Let (\Paramw_1, \ldots, \Paramw_{N_{\uncertpde}})\in \Rbb^{N_{\uncertpde}}\) is the \emph{uncertainty coefficient} vector} (which is to be determined). We rewrite the compatibility condition after the parametrization by
\begin{equation}
    \label{eq:mod_compatibility_condition}
    \inprod{\Param}{\Reg(0)} + \inprod{\Paramw}{\RegD(0)} = \parampde(1).
\end{equation}
%
\noindent In view of the preceding developments, the OCP \eqref{eq:og_OCP} can be rephrased as:
\begin{equation}
	\label{eq:ocp_semi-discrete_parametrized_minmax}
\begin{aligned}
& \hspace{-3mm}\inf_{\contparam(\cdot)} \sup_{\uncertparam(\cdot)}	&& \objectivedisc\bigl(\param, \contparam(\cdot), \uncertparam(\cdot)\bigr) \\
&  \hspace{-3mm}\sbjto		&&  \hspace{-10mm}\begin{cases}
\text{dynamics \eqref{eq:semi-disc}},\,\, \st(\tinit)= \param,\, \contparam(0) = \contpde_0,\\ 
\contparam(0) + \uncertparam(0) = \parampde(1),\, \contparam(t) + \uncertparam(t) \in \admcontpde,\\
\text{for all }\uncertparam(t) \in \uncertsetode \text{ and for all } t \in \lcrc{0}{\horizon}.
\end{cases}
\end{aligned}
\end{equation}
Using the variation of constants formula \cite[\S 2.3]{ref:borzi2020modelling} along with \eqref{e:cont_param} and \eqref{e:uncert_param}, overloading the notation, starting from an initial state \(\st(\tinit) = \param\), the solution of \eqref{eq:semi-disc} for all \(t \in \lcrc{0}{\horizon}\), is
\begin{align}\label{eq:p_sol}
\st(t) &=  e^{A_nt}\param + \int_{0}^{t}e^{A_n(t-\tau)}B_n\left( \inprod{\Param}{\Reg(\tau)}  + \inprod{\Paramw}{\RegD(\tau)}\right)\odif{\tau}.
\end{align}
The cost function in \eqref{eq:discrete_cost} can be rewritten as:
\begin{align}
  \label{eq:parametrized_cost}
       &\objectivedisc\bigl(\param, \inprod{\Param}{\Reg(\cdot)},\inprod{\Paramw}{\RegD(\cdot)} \bigr)  =  \frac{\lambda}{2}\int_{\tinit}^{\horizon} \inprod{\Param}{\Reg(t)}^{2} \odif{t} \nn \\& 
         + \frac{h}{2} \bigg{\|} e^{A_n\horizon}\param + \int_{0}^{\horizon}e^{A_n(\horizon-\tau)}B_n
         \bigl(\inprod{\Param}{\Reg(\tau)} \nn \\& + \inprod{\Paramw}{\RegD(\tau)}\bigr)\odif{\tau}
        - \st_{\spacedomainpde}\bigg{\|}^2  - \frac{\gamma}{2}\int_{\tinit}^{\horizon} \inprod{\Paramw}{\RegD(t)}^2  \odif{t}. 
\end{align}
Substituting \eqref{e:cont_param}, \eqref{e:uncert_param}, and \eqref{eq:parametrized_cost} in \eqref{eq:ocp_semi-discrete_parametrized_minmax}, we obtain the minmax program
\begin{equation}\label{eq:sip_ready_minmax}
\begin{aligned}
& \inf_{\Param} \sup_{\Paramw}	&& \hspace{-2mm} \objectivedisc\bigl(\param,\inprod{\Param}{\Reg(\cdot)},\inprod{\Paramw}{\RegD(\cdot)}\bigr) \\
&  \sbjto		&&  \hspace{-7mm}\begin{cases}
\st(\tinit)= \param,\, \inprod{\Param}{\Reg(0)} = \contpde_0,\\
\inprod{\Param}{\Reg(0)} + \inprod{\Paramw}{\RegD(0)} = \parampde(1), \\ 
\inprod{\Param}{\Reg(t)} + \inprod{\Paramw}{\RegD(t)} \in \admcontpde \text{ for all } t \in \lcrc{0}{\horizon},\\
\text{for all }\inprod{\Paramw}{\RegD(t)} \in \uncertsetode.
\end{cases}
\end{aligned}
\end{equation}  
The next result provides guarantees of nice structure for admissible parameters for the disturbance and control trajectories; its proof is relegated to Appendix \ref{appen:thrm:proofs}.
\begin{myOCP}
\begin{prop}\label{prop:comp:conv:adparam:sets}
Consider the PDE \eqref{eq:parabolic_pde} with boundary conditions \eqref{eq:boundary_condition} and associated data \ref{eq:prob:data:1}–\ref{eq:prob:data:4}. Define the set of admissible parameters corresponding to the disturbance in \eqref{e:uncert_param} by
\begin{align}\label{eq:adm_dist_param}
    \adparamD \Let   \left\{ \Paramw  \;\middle\vert\;  
    \begin{array}{@{}l@{}}
    \contpde_0 + \inprod{\Paramw}{\RegD(0)} = \parampde(1), \blue{\norm{\inprod{\Paramw}{\RegD(\cdot)}}_{\mathcal{R}} \le \adboundw} \\ \inprod{\Paramw}{\RegD(t)} \in \uncertsetode \text{ for all } t \in \lcrc{0}{\horizon}
    \end{array}
        \right\},
\end{align} 
and the control in \eqref{e:cont_param} by
\begin{align}\label{eq:adm_cont_param} 
\adparam \Let   \left\{ \Param \middle\vert 
\begin{array}{@{}l@{}}
\inprod{\Param}{\Reg(0)} = \contpde_0, \blue{\norm{\inprod{\Param}{\Reg(\cdot)}}_{\mathcal{R}} \le \adboundv},\\
\inprod{\Param}{\Reg(t)} + \inprod{\Paramw}{\RegD(t)} \in \admcontpde  \\ \text{ for all } t \in \lcrc{0}{\horizon}, \text{ and for all } \Paramw \in \adparamD
        \end{array}
        \right\}.
    \end{align}
Then sets \(\adparam\) and \(\adparamD\) are compact and convex. 
\end{prop}
\end{myOCP}

\begin{rem}
\blue{We touch upon a point mentioned in the introduction in \ref{numtech:I:DTO}, where we stated that our approach broadly falls under the umbrella of the discretize-then-optimize-type category, although it is quite different from the standard approach. Notice that the ODE-constrained OCP \eqref{eq:sip_ready_minmax} contains uncountably many constraints indexed by both time and disturbance variables, which is nonstandard in existing ODE-constrained optimal control problems. Classical optimal control approaches based on the discretize-then-optimize paradigm \cite{ref:betts-book,ref:Rao-10,ref:kelley,ref:ghobadi2009discretize,ref:Hinze2012,ref:DTO:herzog2023,ref:SG:NR:DC:RB-22,ref:QuITOv2} typically proceed by discretizing the time horizon into a finite, often uniform, grid; discretizing the continuous-time ODE and the cost functional over this grid; and enforcing the constraints only at the discretization points. The resulting finite-dimensional problem is, in general, a nonlinear program. Consequently, the uncountable family of constraints indexed by time over \(\lcrc{0}{T}\) is not handled in a principled manner.

Moreover, the standard literature on ODE-constrained optimal control does not usually address minmax problems involving disturbances. In such problems, the presence of disturbances introduces an additional uncountable family of constraints indexed not only by time but also by all admissible disturbance realizations, which typically take values in compact sets. In contrast, our approach in the sequel reformulates OCP \eqref{eq:sip_ready_minmax} as a convex semi-infinite program and solves it in an \emph{exact fashion}. As a result, all uncountably many continuous-time constraints are enforced directly without directly discretizing the uncertainty variables.
}
\end{rem}

\section{Main results: theory and algorithm}\label{sec:main_result}
We present our main results here in the form of some technical results and an algorithmic architecture. To this end, keeping the optimal value intact, note that \eqref{eq:sip_ready_minmax} can be translated to the minimization problem
\begin{myOCP}
  \begin{equation}
	\label{eq:ocp_semi-discrete_parametrized}
\begin{aligned}
& \hspace{-3mm}\inf_{(\Param, \slack) \in \Rbb^{N_{\contpde}} \times \Rbb}	&& \slack\\
&  \sbjto		&&  \hspace{-3mm}\begin{cases}
\objectivedisc\bigl(\param, \inprod{\Param}{\Reg(\cdot)}, \inprod{\Paramw}{\RegD(\cdot)}\bigr) \leq \slack,\\
\st(\tinit)= \param,\, \inprod{\Param}{\Reg(0)} = \contpde_0,\\ 
\inprod{\Param}{\Reg(0)} + \inprod{\Paramw}{\RegD(0)} = \parampde(1),\\
 \inprod{\Param}{\Reg(t)} + \inprod{\Paramw}{\RegD(t)}\in \admcontpde\\    
\text{for all } t \in \lcrc{0}{\horizon} \text{ and } \Paramw  \in \adparamD.
\end{cases}
\end{aligned}
\end{equation}  
\end{myOCP}
The finite-dimensional problem \eqref{eq:ocp_semi-discrete_parametrized} is a \textcolor{black}{\emph{convex Semi-Infinite Program} (SIP)} since it requires the constraints to be satisfied for all \(t \in \lcrc{0}{\horizon}\) and for all \(\Paramw \in \adparamD\) --- constituting a potentially uncountable family of constraints, which poses significant computational challenges without resorting to conservative approximations, while the program \eqref{eq:ocp_semi-discrete_parametrized} remains finite dimensional. Nonetheless, in the sequel, we present a technique that yields \emph{exact solutions} to \eqref{eq:ocp_semi-discrete_parametrized}, with guarantees of constraint satisfaction.

For the OCP \eqref{eq:ocp_semi-discrete_parametrized}, let us denote the set of feasible initial states by \(\fsblset\). Assuming that \(\fsblset \neq \emptyset\), we denote the value function --- the optimal value of \eqref{eq:ocp_semi-discrete_parametrized} as a function of \(\param\) --- by
\begin{equation}\label{vfunc}
\fsblset \ni \param \mapsto \valuefunc(\param),
\end{equation}
\blue{which is the optimal value of \eqref{eq:ocp_semi-discrete_parametrized}. Following \cite{ref:DasAraCheCha-22}, we introduce the variable below to solve the convex SIP \eqref{eq:ocp_semi-discrete_parametrized}
\begin{align}
\label{dvar}
    \dvar &\Let  N_{\contpde} +1 
\end{align} 
where the first term is the dimension of \(\alpha\) and the second one is the dimension of the slack variable \(s\); thus \(\widetilde{N}\) captures the dimension of decision space of \eqref{eq:ocp_semi-discrete_parametrized}. Using the notation introduced above, we stipulate the following technical assumption which is necessary for our technical results.}

\begin{assum}
\label{assum:slater_cond}
For the OCP \eqref{eq:ocp_semi-discrete_parametrized}, the following Slater-type condition holds: for every \(\dvar\mbox{-}\)tuple \(\bigl(t^i, \uncertparam^i(\cdot) \bigr)_{i=1}^{\dvar}\), the intersection
\begin{align*}
        \bigcap_{j=1}^{\dvar}\hspace{-1mm}
        \left\{\hspace{-1mm}\bigl(\slack,\Param\bigr)\in \Rbb \times \Rbb^{N_{\contpde}}\middle\vert 
\begin{array}{@{}l@{}}
\objectivedisc\bigl(\param, \contparam(\cdot), \uncertparam^i(\cdot)\bigr) < \slack,\,\st(\tinit)= \param\\ \contparam(0) = \contpde_0, \contparam(0) + \uncertparam^i(0)  = \parampde(1),\\
\contparam(t^i) + \uncertparam^i(t^i) \in \admcontpde', \text{where }\\
\bigl(t^i, \uncertparam^i(\cdot)\bigr) \in \lcrc{0}{\horizon} \times \dict_{\uncertpde}\\ \text{ for each } i \in \aset[]{1,\ldots,\dvar}.
        \end{array}
        \right\}
    \end{align*}
    is nonempty, where \(\admcontpde' \subset \intr{\admcontpde}.\)
\end{assum}

\subsection{Our key theoretical result}
We now focus on solving the SIP \eqref{eq:ocp_semi-discrete_parametrized}. Let \(\param \in \fsblset\) be fixed. We define the sequences
\begin{equation}\label{eq:param_sequence}
   \hspace{-3mm} \tseq \Let (t^1,\ldots,t^{\dvar}) \in \lcrc{0}{\horizon}^{\dvar},\, \bseq \Let (\Paramw^1,\ldots,\Paramw^{\dvar}) \in \adparamD^{\dvar},
\end{equation}
and we define the function \(\lcrc{0}{\horizon}^{\dvar} \times \adparamD^{\dvar} \ni  \bigl(\tseq, \bseq\bigr) \mapsto \gfunc(\tseq, \bseq; \param) \in \Rbb\) given by
\begin{equation}
    \label{eq:g_func}
    \begin{aligned}
        &\gfunc(\tseq, \bseq; \param) && \Let\\  &\inf_{(\slack,\Param)\in \Rbb \times \adparam}	&& \slack \\
        &  \sbjto &&  \hspace{-4mm}\begin{cases}
\objectivedisc\bigl(\param, \inprod{\Param}{\Reg(\cdot)}, \inprod{\Paramw^i}{\RegD(\cdot)} \bigr) \leq \slack,\,\, \st(\tinit)= \param,\\
\inprod{\Param}{\Reg(0)} = \contpde_0,\,
 \inprod{\Param}{\Reg(t^i)} + \inprod{\Paramw^i}{\RegD(t^i)}\in \admcontpde\\ \text{for each }i \in \aset[]{1,\ldots,\dvar}.
        \end{cases}
    \end{aligned}
\end{equation}
 For any given \(\bigl(\tseq, \bseq\bigr)\), the optimization problem \eqref{eq:g_func} serves as a relaxed counterpart to \eqref{eq:ocp_semi-discrete_parametrized}. We now record one of the main result of this article.
\begin{myOCP}
\begin{theorem}\label{thrm:value_func_equality}
Consider the OCP \eqref{eq:og_OCP} and suppose that Assumption \ref{assum:slater_cond} is in force. Consider also the SIP \eqref{eq:ocp_semi-discrete_parametrized} and \eqref{eq:g_func} along with its associated data and notations, and define \(\totaluncertset \Let \lcrc{0}{\horizon}^{\dvar} \times \adparamD^{\dvar}\). Recall the notation established in \eqref{eq:param_sequence}, fix \(\param\in \fsblset\), and consider the global maximization problem 
    \begin{equation}
    \label{e:global_max_prob}
    \begin{aligned}
        & \sup_{(\tseq, \bseq)}
        &&\gfunc(\tseq, \bseq; \param)\\
        & \sbjto   && \bigl(\tseq, \bseq\bigr) \in \totaluncertset.
    \end{aligned}
\end{equation}
    Then:
\begin{enumerate}[label=\textup{(\ref{thrm:value_func_equality}-\alph*)}, leftmargin=*, widest=b, align=left]
\item \label{thrm:value_func_equality_0} (regularity) \(\totaluncertset \ni  \bigl(\tseq, \bseq\bigr) \mapsto \gfunc(\tseq, \bseq; \param) \in \Rbb\) is Lipschitz continuous for every \(\param \in \fsblset\),
\item (existence) \label{thrm:value_func_equality_1} there exists  \(\bigl(\tseq\as(\param), \bseq\as(\param)\bigr) \in \totaluncertset\) that solves \eqref{e:global_max_prob}, and 
\item (exact solutions) \label{thrm:value_func_equality_2} \(\valuefunc(\param)=\gfunc\bigl(\tseq\as(\param), \bseq\as(\param);\param\bigr)\) for every \(\param \in \fsblset\).
\end{enumerate}
\end{theorem}
\end{myOCP}
A proof of Theorem \ref{thrm:value_func_equality} is given in Appendix \ref{appen:thrm:proofs}.

The schematic in Fig.\ \ref{fig:flowchat_PdeDenCon} depicts a bird's-eye view of our approach. 

\begin{figure}[h]
\centering
\begin{tikzpicture}[
  scale=0.5,
  node distance=0.7cm and 0.7cm,
  block/.style={
    draw, 
    rounded corners,
    minimum width=4cm,    
    minimum height=2cm,   
    align=center,
    fill=#1,
    fill opacity=0.8,
    text opacity=1,
    draw opacity=0.1
  }
]
\node[block=Orchid!25!white] (A) {PDE-constrained \\ minmax OCP \eqref{eq:og_OCP}};
\node[block=Bittersweet!20!white, below=of A] (B) {ODE-constrained \\ minmax OCP \eqref{eq:sip_ready_minmax} \\ with infinitely\\ many constraints};
\node[block=CornflowerBlue!20!white, right=of B] (C) {Convex semi-infinite \\ program \eqref{eq:ocp_semi-discrete_parametrized}};
\node[block=Goldenrod!20!white, below=of C] (D) {Finite relaxation \\ \eqref{eq:g_func}};
\node[block=LimeGreen!20!white, below=of B] (E) {Guarantees of exact \\ solutions, plug-and-play \\ global optimization module};

\draw[->, line width=2pt] (A) -- (B);
\draw[->, line width=2pt] (B) -- (C);
\draw[->, line width=2pt] (C) -- (D);
\draw[->, line width=2pt] (D) -- (E);

\end{tikzpicture}
\caption{Hierarchical transformation of a PDE-constrained minmax OCP to a finite convex relaxation with guarantees of recovering the true optimal solution.}
\label{fig:flowchat_PdeDenCon}
\end{figure}
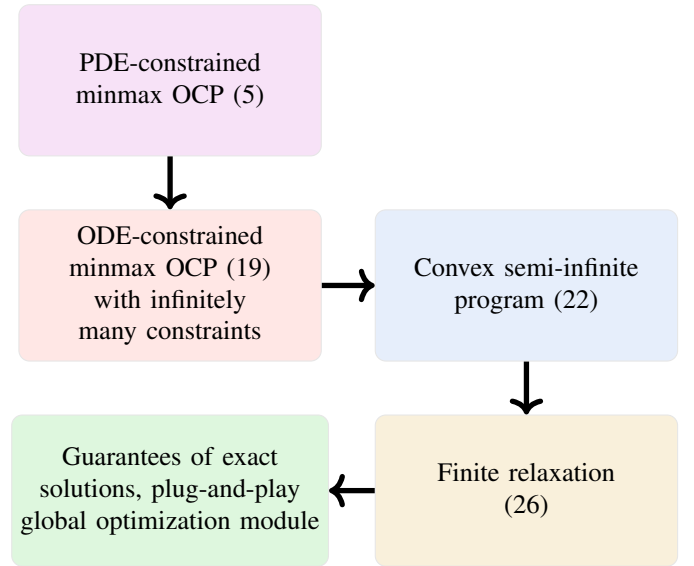


\begin{rem}
\label{rem:on_reformulation}
The assertion in \ref{thrm:value_func_equality_0} establishes that the mapping \(\totaluncertset \ni (\tseq, \bseq) \mapsto \gfunc(\tseq, \bseq; \param) \in \Rbb\) is well-behaved. Consequently, solutions to \eqref{e:global_max_prob} exist, and one can readily employ efficient global optimization algorithms with convergence guarantees to solve \eqref{e:global_max_prob}, owing to the Lipschitz continuity of the mapping \(\gfunc(\cdot,\cdot; \param)\).
The assertion \ref{thrm:value_func_equality_2} establishes that it is sufficient to consider only (intelligently picked) \(\dvar\)-many constraints at which the constraints must be satisfied to solve the SIP \eqref{eq:ocp_semi-discrete_parametrized} with the guarantee that the optimal value of \eqref{eq:ocp_semi-discrete_parametrized}, is the same as the optimal value of \eqref{eq:g_func}. To recover such an optimal point \(\bigl(\tseq^{\ast}, \bseq^{\ast}\bigr)\), the maximization problem in Theorem \ref{thrm:value_func_equality} must be solved globally on the set \(\totaluncertset.\) The architecture \(\glbopt\) in \S\ref{sec:algorithm} ahead formalizes the steps involved in solving \eqref{eq:ocp_semi-discrete_parametrized}.
\end{rem}

\begin{rem}\label{rem:on:nonconvex:SIPs}
We considered the heat equation --- a \emph{linear} parabolic PDE --- since its semi-discretization results in a \emph{linear} system of ODEs. In contrast, for a general class of nonlinear parabolic systems, spatial semi-discretization leads to a nonlinear ODE in time, resulting in a nonconvex SIP. While there are numerical techniques available for solving nonconvex SIPs, to the best of our knowledge, no method guarantees exact solutions while remaining finitary, and the required number of realizations of the uncertainty parameters remains undetermined, except for the assurance that the algorithm is asymptotically optimal. Nevertheless, there are efficient methods for addressing nonlinear and nonconvex SIPs, such as those presented in \cite{ref:okuno2023primal, ref:okuno2020interior} and \cite{ref:Stein2020}. These approaches will be explored in our future work.
\end{rem}

\subsection{An architecture to solve the global optimization problem \eqref{e:global_max_prob}}\label{sec:algorithm}

We introduce an algorithmic architecture, denoted by \(\glbopt\), designed to solve the global optimization problem stated in \eqref{e:global_max_prob}. The \(\glbopt\) architecture provides a unified and flexible structure for synthesizing constrained control trajectories by employing a suitable global optimization routine tailored to the characteristics of the function \(\gfunc(\cdot,\cdot;\param)\).

{
\renewcommand{\algorithmcfname}{\(\glbopt(\param)\)}
\renewcommand{\thealgocf}{}
\begin{algorithm2e}[!ht]
\DontPrintSemicolon
\SetKwInOut{ini}{Initialize}
\SetKwInOut{giv}{Data}
\giv{Stopping criterion $\SC(\cdot)$, threshold for the stopping criterion \(\tau\), fix \(\param \in \fsblset.\)}
\ini{initialize the constraint indices \( \bigl(\tseq^{\text{in}}, \bseq^{\text{in}}\bigr)\Let \Bigl(\bigl(t^{\text{in},i}\bigr)_{i = 1}^{\dvar}, \bigl(\uncertparam^{\text{in},i}\bigr)_{i = 1}^{\dvar} \Bigr) \in \totaluncertset\), initial guess for \(\gfunc_{\text{max}}\), initial guess for the solution \(\overline{v}_{\Dd}\)}
   
\While{$\SC(m) \leqslant \tau$}
{

Sample (via any global optimization method) the disturbance constraint set \(\bigl(\tseq^{m}, \bseq^{m}\bigr)\Let \Bigl(\bigl(t^{m,i}\bigr)_{i = 1}^{\dvar}, \bigl(\uncertparam^{m,i}\bigr)_{i = 1}^{\dvar} \Bigr) \in \totaluncertset\)
           
Solve the \emph{inner-problem} and evaluate \(\gfunc_m = \gfunc\bigl(\tseq^m, \bseq^m;\param\bigr) \) as defined in \eqref{eq:g_func}

\emph{Recover} the solution  \(\contparam^m,\slack^m \in\)
\begin{align*}
\argmin_{(\widetilde{v}_{\Dd},\widetilde{\slack})}\left\{\widetilde{\slack} \middle\vert 
\begin{array}{@{}l@{}}
\text{ constraints in } \eqref{eq:g_func} \text{ hold at } \bigl(\tseq^m, \bseq^m\bigr)
\end{array}
\right\}
\end{align*}

Update \( (\gfunc_{\text{max}}, \overline{v}_{\Dd})\) \(\longleftarrow\) \textsf{IC}\((\gfunc_m, \contparam^m)\)    
            
Update \(m \gets m+1\) \;
}

\caption{A general architecture to solve \eqref{e:global_max_prob}}
\label{alg:sip_algo}
\end{algorithm2e}
}

\begin{itemize}[leftmargin=*,label=\(\circ\)]
    \item Given the (Lipshitz) continuity of \(\gfunc(\cdot,\cdot;\param)\) established in Theorem \ref{thrm:value_func_equality}, the convergence guarantees of both the simulated annealing and differential evolution oracles hold within the \(\glbopt\) framework under standard assumptions such as a suitable cooling schedule and mild regularity conditions on the transition kernel, as shown in \cite[Theorem 2]{ref:DasAraCheCha-22}, \cite[Theorem 1]{ref:CJPB-92}, and \cite{ref:storn1997differential}. 

    \item Efficient global optimization methods such as \texttt{SequOOL} (Sequential Optimistic Optimization with Levels) \cite{ref:BarGabVal-19} and \texttt{LIPO} (Lipschitz Optimization based on Local Partitions) \cite{ref:MalVay-17} can also be employed to solve \eqref{e:global_max_prob} since \(\gfunc(\cdot,\cdot;\param)\) is Lipschtiz. \texttt{SequOOL} uses a deterministic hierarchical strategy with exponential regret bounds \cite{ref:GriValMun-15}, while \texttt{LIPO} employs stochastic averaging with PAC-style guarantees. Both converge efficiently in practice, especially with randomized restarts.

    \item Several hyperparameter in \(\glbopt\) can be picked depending upon the type of applications. The loop uses a global optimization routine with a stopping criterion \(\textsf{SC}\), such as a maximum number of iterations, a temperature threshold, or a bound on successive value differences (e.g., in simulated annealing). \(\textsf{IC}\) is the Selection and Improvement Criterion of the global optimization routine, determining whether to accept a candidate solution. In simulated annealing, a candidate is accepted if it improves the current best, or otherwise with probability \(\exp{\left(\frac{\gfunc^{m}-\gfunc_{\max}}{T_m}\right)}\), where \(T_m\) is the temperature.
\end{itemize}

\begin{rem}\label{rem:on:general:systems:semi-disc}
    \textcolor{black}{Our result and the algorithmic architecture established in this section can also be extended to more general linear parabolic PDEs of the form
    \begin{equation}\label{eq:parabolic_pde_gen}
        \stpde_t(\spacepde, t) = \theta(\spacepde)\stpde_{\spacepde \spacepde}(\spacepde, t) + \rho(\spacepde)\stpde_{\spacepde}(\spacepde, t) + \sigma(\spacepde)\stpde(\spacepde, t),
    \end{equation}
    subject to the boundary conditions
    \(\alpha_0 \stpde_{\spacepde}(0, t) + \beta_0 \stpde(0, t) = 0,\) \(\alpha_1 \stpde_{\spacepde}(1, t) + \beta_1 \stpde(1, t) = \contpde(t) + \uncertpde(t),\) for all \(t \in \lcrc{0}{\horizon}\), where the coefficients \(\theta, \rho, \sigma \in \text{PC}^1\bigl(\lcrc{0}{1}; \Rbb\bigr)\), and \(\theta(\cdot)\) is uniformly positive over the domain, i.e., \(\inf_{\spacepde \in \lcrc{0}{1}} \theta(\spacepde) > 0\).\footnote{A function \(\psi : \lcrc{0}{1} \to \Rbb\) is in \(\text{PC}^1(\lcrc{0}{1}; \Rbb)\), if there exists a finite partition of \(\lcrc{0}{1}\), say \(0 \teL t_0 < t_1 < t_2 < \ldots < t_{k-1} < t_k \Let 1\), such that \(\psi(\cdot) \in \conti{1}\bigl(\lcrc{t_i}{t_{i+1}}; \Rbb\bigr)\) for all \(i \in \{0, \ldots, k-1\}\), meaning that \(\psi(\cdot) \in \conti{1}\bigl(\loro{t_i}{t_{i+1}}; \Rbb\bigr)\) and the right and left derivatives exist at \(t_i\) and \(t_{i+1}\), respectively.} This class of systems, under some mild assumptions, also admits the linear ODE structure in \eqref{eq:semi-disc} after semi-discretization, and similar convergence guarantees as above can be given; see \cite{ref:SC:VN:motion:TAC25}. For clarity of exposition, we focussed on the PDE system \eqref{eq:parabolic_pde}--\eqref{eq:boundary_condition}.
}
\end{rem}


\section{Numerical experiments}\label{sec:numexp}
\begin{figure*}[h]
  \begin{subfigure}[b]{0.32\linewidth}
    \includegraphics[width=6cm,height=5cm]{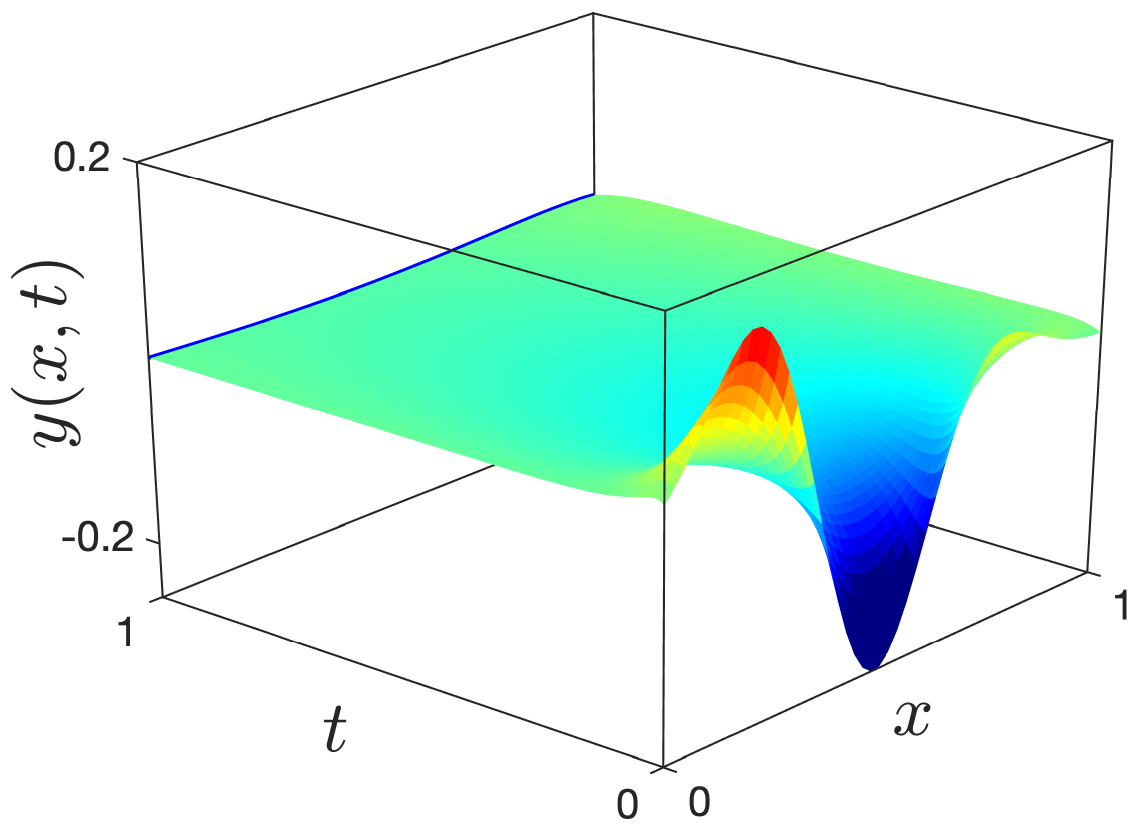}
  \end{subfigure}
  \begin{subfigure}[b]{0.32\linewidth}
    \includegraphics[width=6cm,height=5cm]{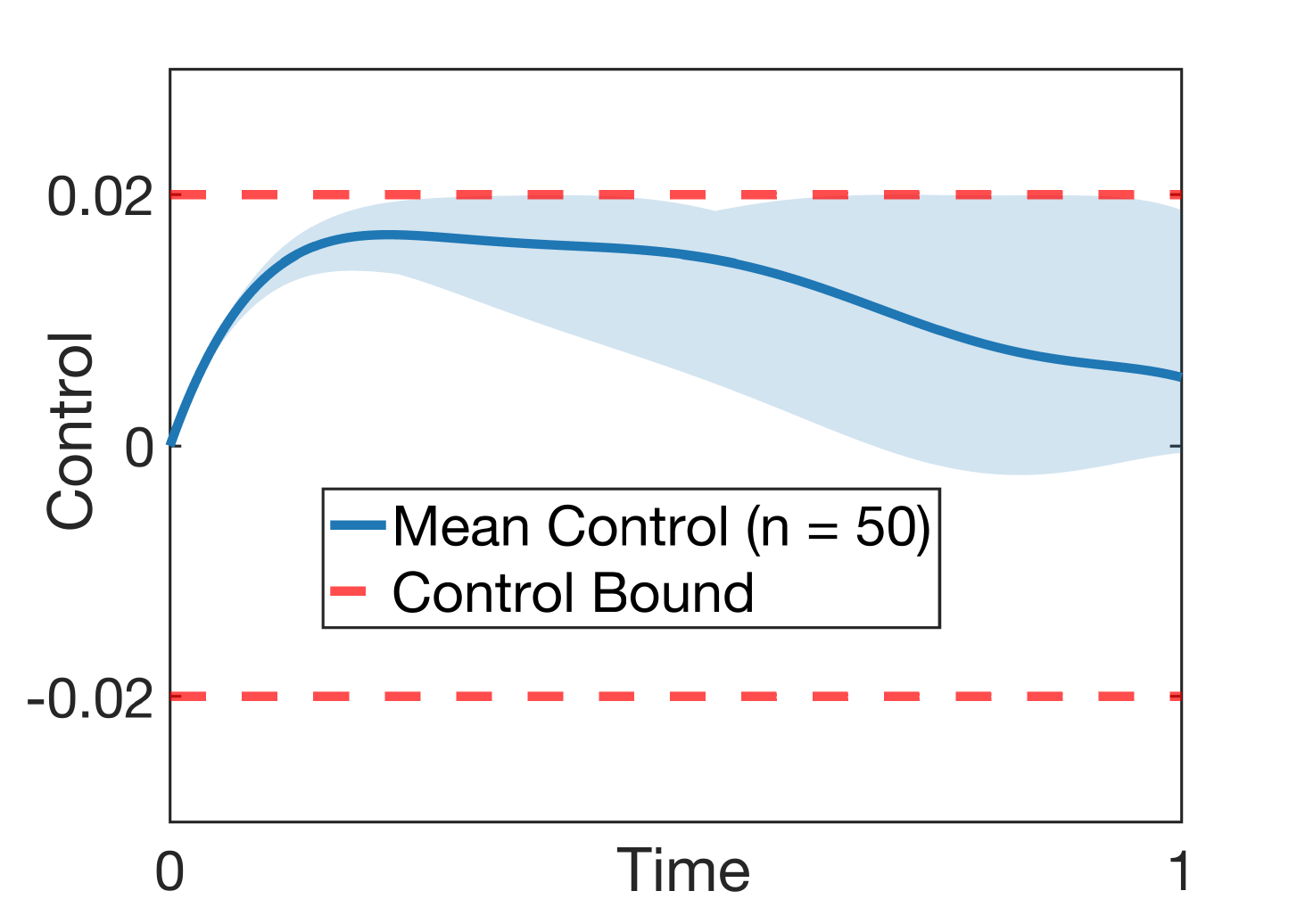}
  \end{subfigure}
  \begin{subfigure}[b]{0.3\linewidth}
    \includegraphics[width=6cm,height=5cm]{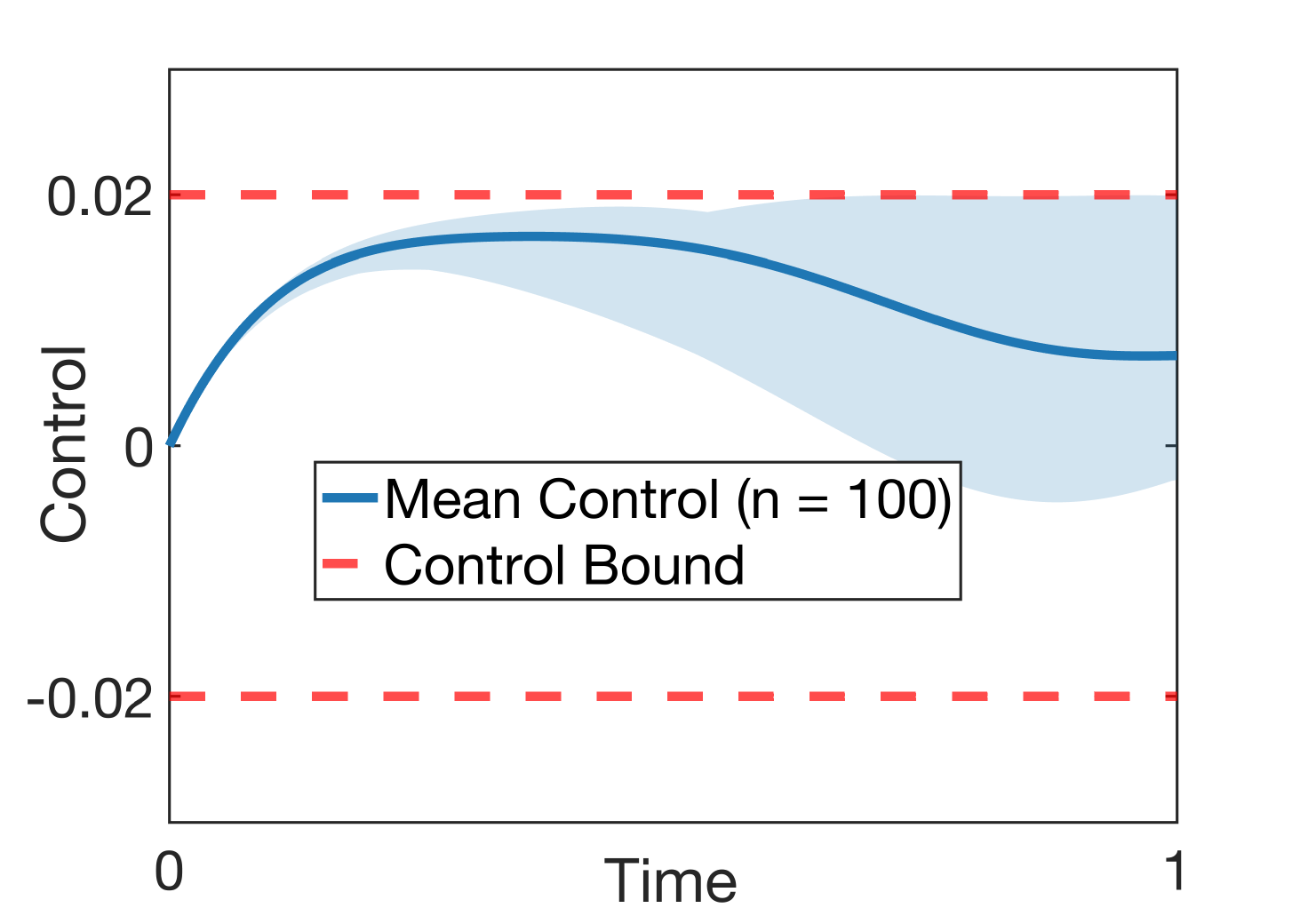}
  \end{subfigure}
\caption{\textcolor{black}{Steering landscapes corresponding to the subproblems \textbf{P1} (left-hand subfigure). The initial and the final states are given in \eqref{num:D2Z:init:targ}. The control successfully guided the initial states to their target states. The middle and right-hand subfigure shows the evolution of the composite control input \(t \mapsto v(t) + w(t)\) for discretization level \(n = 50\) and \(n=100\) respectively, corresponding to Problem \textbf{P1} in Example \ref{numexp:heat_eq:example}. The shaded envelope around the mean control represents the band of deviation resulting from boundary uncertainty. To compute this deviation, the uncertainty parameters \(\Paramw\) correspond to the maximizers of \( \gfunc(\cdot,\cdot;\param)\) in equation \eqref{e:global_max_prob}. The entire diagram illustrates the overall control under uncertainty, constrained by the imposed conditions.}
}
\label{fig:state:dist:D2N:D2D:SS}
\end{figure*}

\begin{figure*}[h]
  \begin{subfigure}[b]{0.32\linewidth}
    \includegraphics[width=6cm,height=5cm]{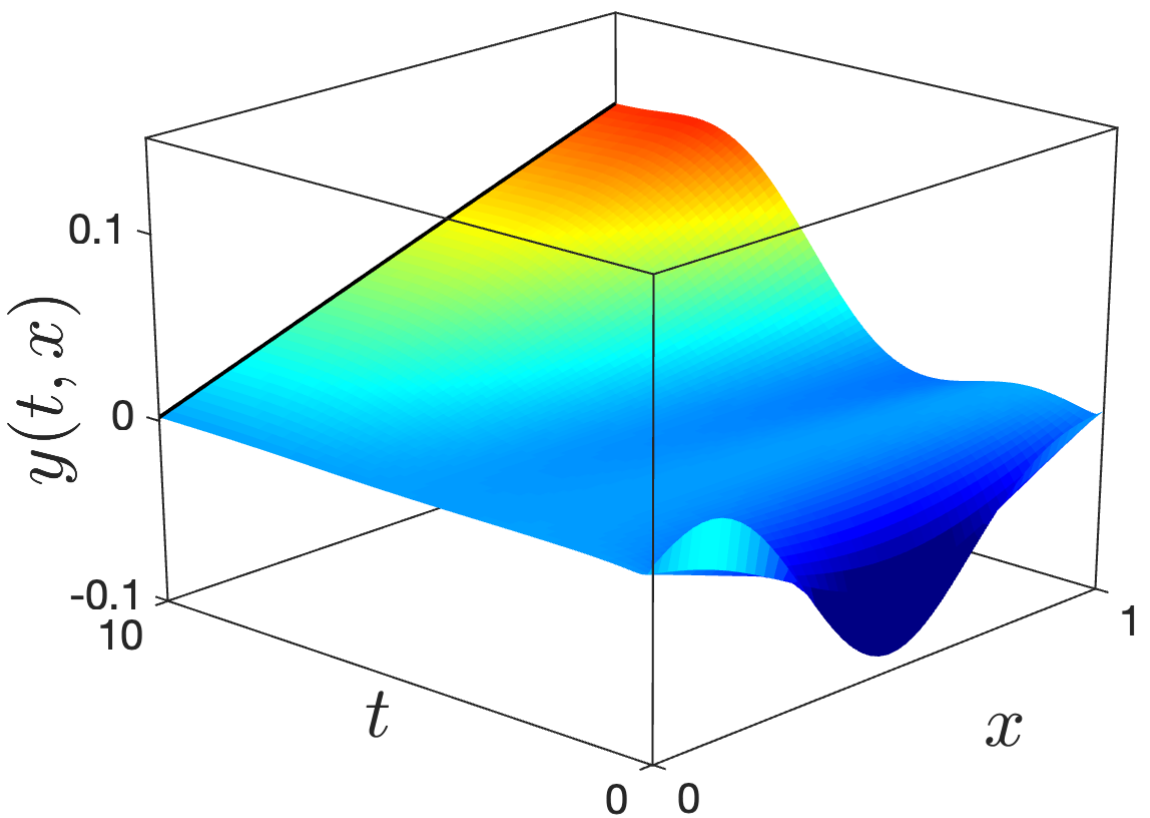}
  \end{subfigure}
  \begin{subfigure}[b]{0.32\linewidth}
    \includegraphics[width=6cm,height=5cm]{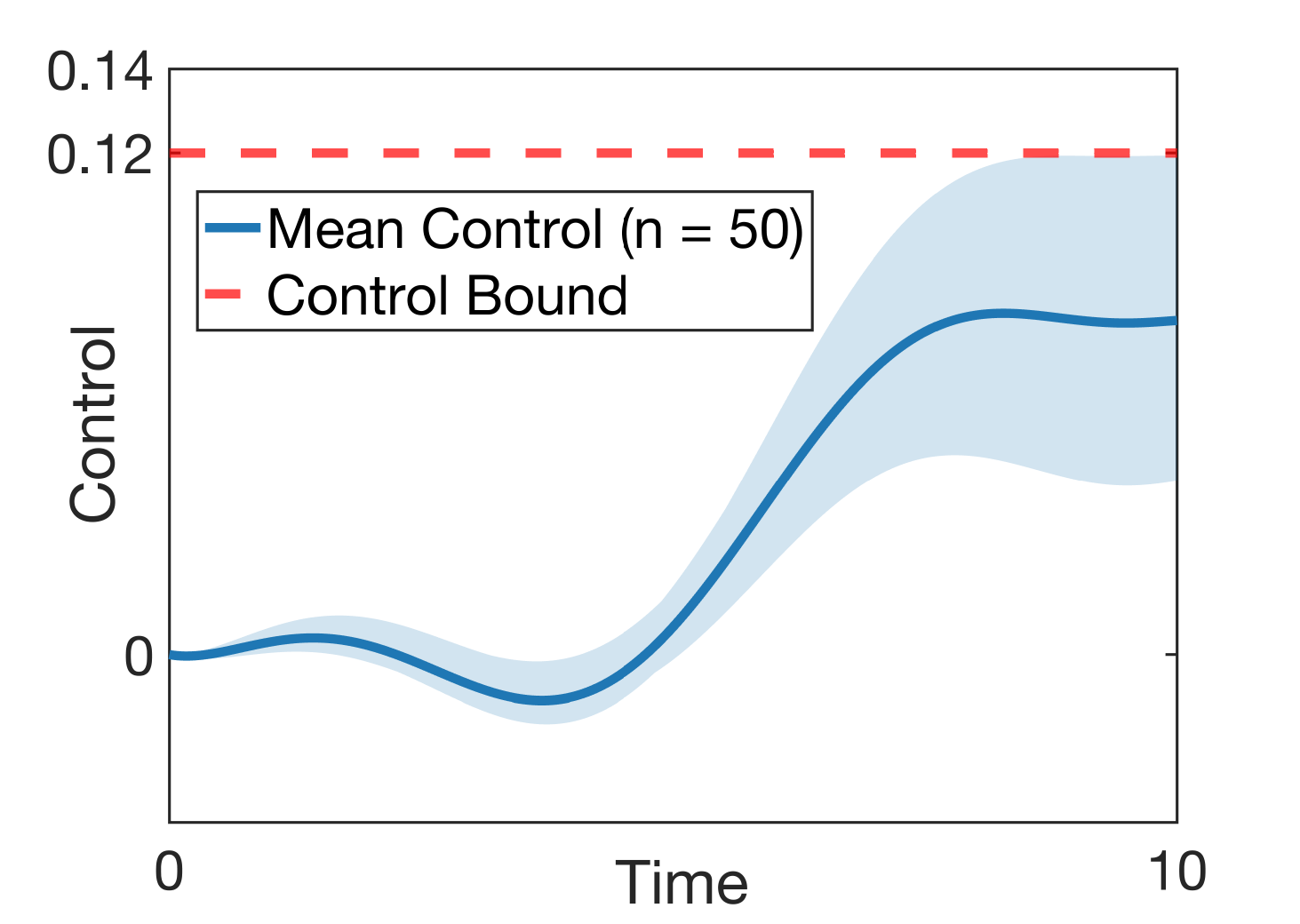}
  \end{subfigure}
  \begin{subfigure}[b]{0.3\linewidth}
    \includegraphics[width=6cm,height=5cm]{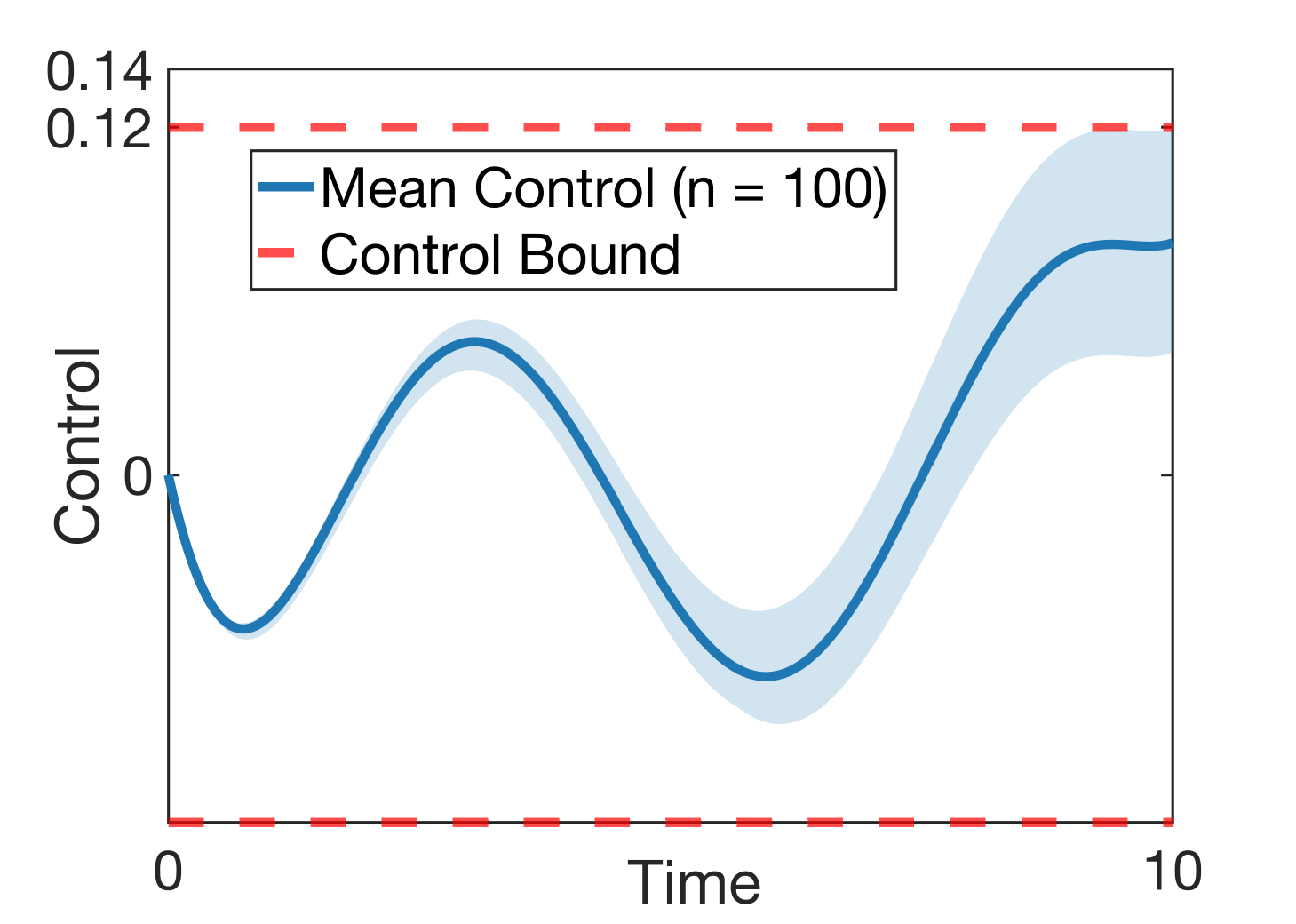}
  \end{subfigure}
\caption{\textcolor{black}{Steering landscapes corresponding to the subproblems \textbf{P2} (left-hand subfigure). The initial and the final states are given in \eqref{num:D2D:SS:init:targ}. The control successfully guided the initial states to their target states. The middle and right-hand subfigure shows the evolution of the composite control input \(t \mapsto v(t) + w(t)\) for discretization level \(n = 50\) and \(n=100\) respectively, corresponding to Problem \textbf{P2} in Example \ref{numexp:heat_eq:example}.}
}\label{fig:control:dist:D2N:D2D:SS}
\end{figure*}  
We provide a couple of numerical results to illustrate the performance of the architecture \(\glbopt\). Reproducible codes may be found in the GitHub repository: \url{https://github.com/v-upadhyay/GloSolPDE.git}.   

\subsection{Heat equation}\label{numexp:heat_eq:example}
Define \(\Gamma = \loro{0}{1} \times \lcrc{0}{T}\) and consider the heat equations accompanied by Dirichlet boundary conditions
\begin{equation}
    \label{eq:heat_eqn_dirtch}
    \stpde_t(\spacepde,t) = c \stpde_{\spacepde\spacepde}(\spacepde,t) \,\,\text{for all }(\spacepde,t) \in \Gamma,
\end{equation}
\blue{where \(c = 0.188\) denotes the thermal diffusivity for steel (in \({\text{cm}^2}/{\text{s}}\)) \cite[Table A.1, p.714]{ref:LHD-19},} with the following boundary conditions\footnote{\blue{Without loss of generality, here we consider the example with a non-unit diffusion coefficient unlike \eqref{eq:parabolic_pde}. Indeed, for some \(T_0 \in \N\) and a fixed constant coefficient \(c>0\), starting from \(y_t=cy_{xx}\) on \(\lcrc{0}{T_0}\), the
time rescaling \(\tau \mapsto ct\) yields \(y_{\tau}=y_{xx}\) on \(\lcrc{0}{cT_0}\).}}
\begin{equation}
    \label{eq:dirichlet_boundary}
     \stpde(0, t) = 0, \,\,\,  \stpde(1, t) = \contpde(t) + \uncertpde(t) \,\,\, \text{for all  } t \in \lcrc{0}{\horizon}.
\end{equation}
The initial and desired terminal states, for all \(\spacepde \in \loro{0}{1}\), are
\begin{equation}\label{eq:num:init_final_conditions}
   x \mapsto y(x,0)\Let \bar{y}(\spacepde) \quad\text{and} \quad x \mapsto y(x,\horizon)\Let \terminalstpde(\spacepde), \nn
\end{equation}
which will be specified below. With the above ingredients, consider the OCP
\begin{equation}
	\label{eq:heat_ocp}
\begin{aligned}
& \inf_{\contpde(\cdot) }	\sup_{\uncertpde(\cdot) }&&  \frac{1}{2} \int_{0}^{1} |\stpde(x,T) - \terminalstpde(x)|^2 \odif{x} \nn \\ 
& && + \frac{\lambda}{2}\int_{\tinit}^{\horizon} \contpde(t)^{2} \odif{t} - \frac{\gamma}{2} \int_{\tinit}^{\horizon} \uncertpde(t)^2 \odif{t} \nn\\
&  \sbjto		&&  \begin{cases}
    \eqref{eq:heat_eqn_dirtch},\, \stpde(\spacepde, \tinit)= \bar{\stpde}(\spacepde),\,\eqref{eq:dirichlet_boundary}, \ref{eq:prob:data:3},\,\contpde(0)= \contpde_0\\
    \abs{\contpde(t) + \uncertpde(t)} \le \mu_\contpde,\,\, \text{for all\ } t \in \lcrc{0}{\horizon},\\
    \text{for all } \abs{\uncertpde(t)} \le \mu_\uncertpde.
\end{cases}
\end{aligned}
\end{equation}
We will consider two different cases of state steering --- (a) from a given initial state \(x\mapsto \bar{y}(x)\) to the null state \(x\mapsto y_{\Omega}(x)=0\) (b) from an arbitrary initial state \(x \mapsto \bar{y}(x)\) to a terminal state \(x \mapsto y_{\Omega}(x)\). We solve the ensuing SIPs (arising from the minmax problems) using the architecture \(\glbopt\) presented in \S\ref{sec:algorithm}. \blue{For the global maximization of \(\gfunc(\cdot,\cdot; \param)\) in the architecture \(\glbopt\), we employed the gradient-based algorithm \(\gradol\) \cite{ref:AR_MB_DC-26} using the \texttt{BlackBoxOptim.jl} library in Julia.}\footnote{\blue{\(\gradol\) is a gradient-based method for finding the smallest enclosing ball of a set (the Chebyshev center) and, more generally, for solving convex semi-infinite optimization problems.}} To solve the relaxed inner optimization problem within the same architecture, we used the nonlinear interior-point solver \texttt{IPOPT} \cite{ref:IPOPTwachter2006implementation} through the \texttt{JuMP} modeling interface in Julia 1.10.0 \cite{ref:Julia:bezanson2017}, with a maximum iteration limit of \( N_{\text{iter}} = 100 \). We provide two cases where the task of our controller is to steer an initial state \(\bar{y}(\cdot)\) to a terminal state \(y_{\Omega}(\cdot)\). We first set \(N_{\uncertpde} = 20\) and use sinusoidal basis functions to parametrize the disturbance trajectory, defined by
\begin{equation}
\begin{aligned}
    \RegD_i(t) \Let \begin{cases}
        1 & \hspace{-2mm}\text{if } i=1,\\
        \sin(2\pi(i-1)t) & \hspace{-2mm}\text{if }i \in \{2,\dots, \tfrac{N_{\uncertpde}+1}{2}\},\\
        \cos\Big(2\pi \Big(i-\frac{N_{\uncertpde}+1}{2}\Big)t\Big) & \hspace{-2mm}\text{if } i \in \{\tfrac{N_{\uncertpde}+3}{2},\dots,N_{\uncertpde}\},\nn
    \end{cases}
\end{aligned}
\end{equation}
and then adopt the same basis for the control trajectory by setting \(N_{\contpde} = N_{\uncertpde}\) and \(\Reg_i(\cdot) = \RegD_i(\cdot)\) for all \(i \in \aset[]{1, \ldots, N_{\uncertpde}}\).
\begin{itemize}[leftmargin=*]
\item \textbf{P1: Steering a given state to the null state:} The initial and final states and the problem parameters are
\begin{align}\label{num:D2Z:init:targ}
x \in \loro{0}{1} \mapsto
\begin{cases}
\bar{y}(x) \Let 10 x^2 (1 - x)^3 \sin(3\pi x), \\
\terminalstpde(x) \Let 0,
\end{cases}
\end{align}
\textcolor{black}{\((\horizon,\lambda, \gamma, \mu_\contpde, \mu_\uncertpde) \Let (1, 10^{-5}, 10^{-3}, 0.02, 0.001)\)}.


\item \textbf{P2: Steering a given state to steady state:} The initial and final states and the problem parameters are
\begin{align}\label{num:D2D:SS:init:targ}
\loro{0}{1} \ni x \mapsto
\begin{cases}
\bar{y}(x) \Let 10 x^2 (1 - x)^3 \sin(3\pi x), \\
\terminalstpde(x) \Let \frac{x}{10},
\end{cases}
\end{align}
\((\horizon,\lambda, \gamma, \mu_\contpde, \mu_\uncertpde) \Let (10, 10^{-5}, 10^{-3}, 0.12, 0.003)\).
\begin{figure*}[h]
  \begin{subfigure}[b]{0.32\linewidth}
    \includegraphics[width=5.5cm,height=4.5cm]{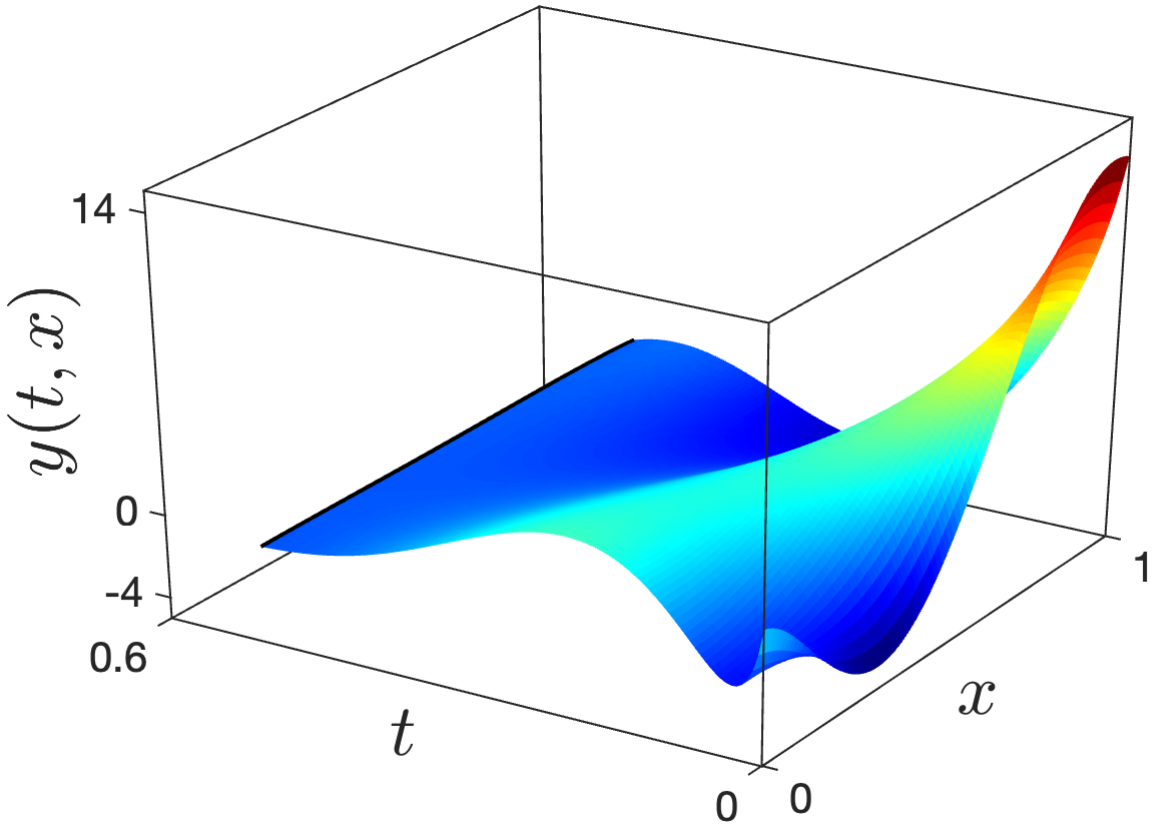}
  \end{subfigure}
  \begin{subfigure}[b]{0.32\linewidth}
    \includegraphics[width=6cm,height=5cm]{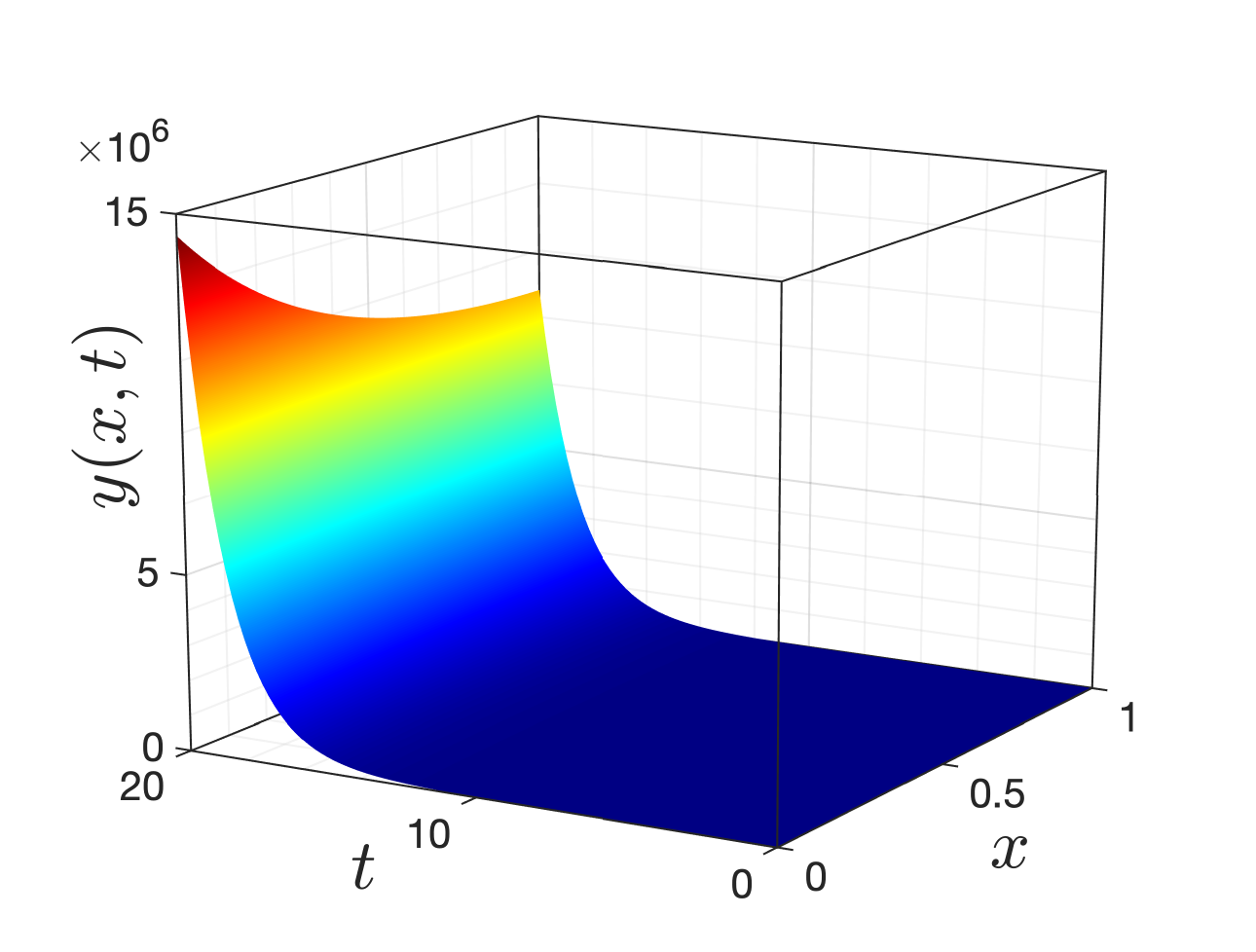}
  \end{subfigure}
  \begin{subfigure}[b]{0.3\linewidth}
    \includegraphics[width=6cm,height=5cm]{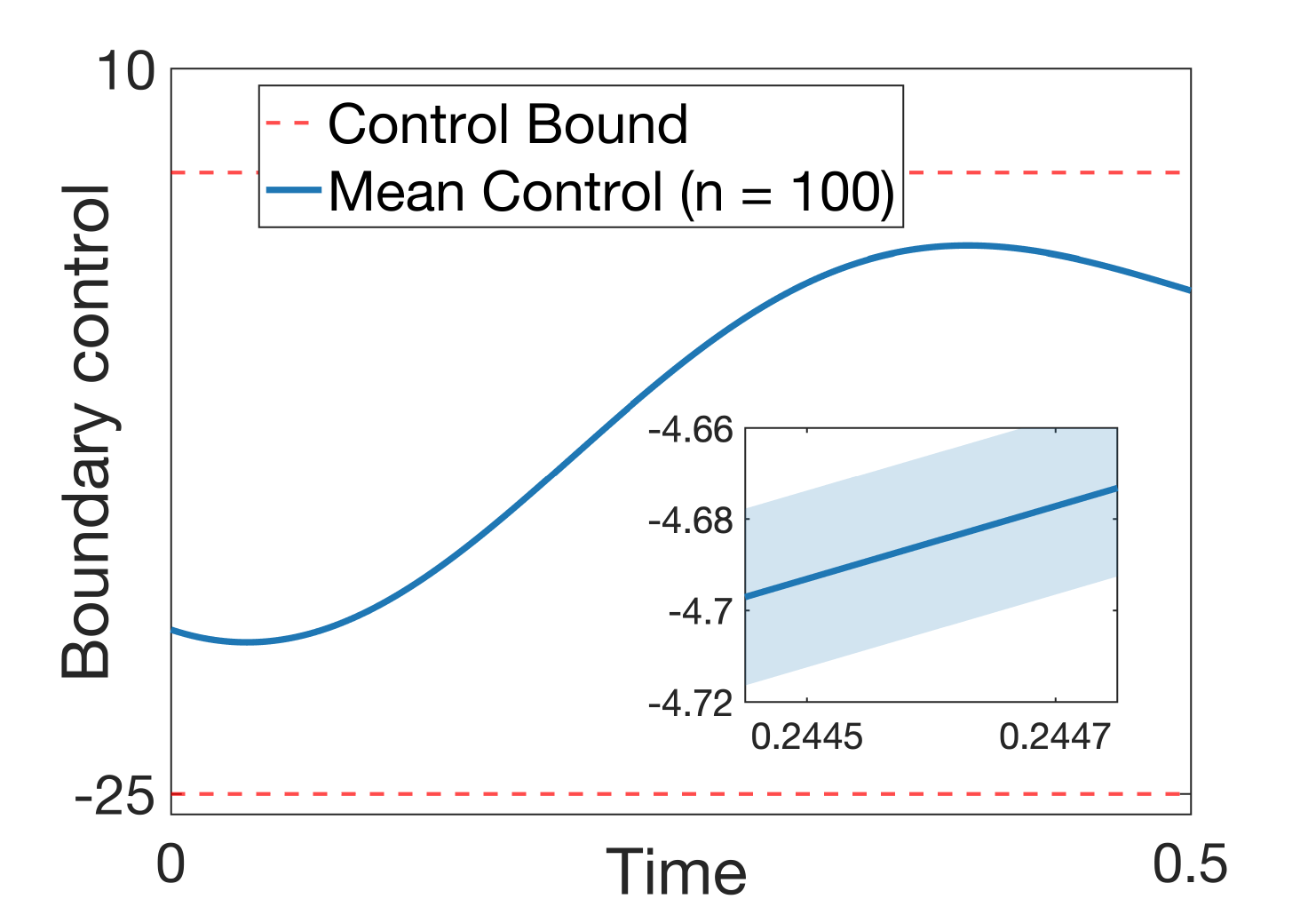}
  \end{subfigure}
\caption{The subfigure in middle depicts the uncontrolled state distribution of the system \eqref{eq:rad_robin}--\eqref{eq:robin_boundary}, highlighting its inherent instability. In contrast, the left-hand subfigure shows that, under the control input generated by our architecture, the distribution is successfully steered to the specified target state. The right-hand subfigure shows the evolution of \(t \mapsto v(t) + w(t) \) for discretization level \(100\). The shaded envelope around the mean control illustrates the band of deviation arising from boundary uncertainty, computed using the parameters \( \Paramw \) that maximize \( \gfunc(\cdot, \cdot; \param) \) in equation \eqref{e:global_max_prob}, thus capturing the worst-case impact on control performance.
}\label{fig:state:dist:Ex2:D2N}
\end{figure*}
\end{itemize}
The state distributions resulting from the solution of the OCP \eqref{eq:heat_ocp} using architecture \(\glbopt\) \blue{(with \(\gradol\))} for both the itemized scenarios are shown in Fig.\ \ref{fig:state:dist:D2N:D2D:SS} (the left-hand and the middle subfigures). It can be observed that, despite the presence of uncertainties, \(\glbopt\) effectively steers \(\bar{y}(\cdot)\) towards the desired terminal state \(y_{\Omega}(\cdot)\) in all three cases. The corresponding boundary control trajectories, computed for different spatial discretization levels \(n=50,100\), are shown in Fig. \ref{fig:state:dist:D2N:D2D:SS} (right-hand subfigure) and Fig. \ref{fig:control:dist:D2N:D2D:SS} (left-hand subfigure), respectively, demonstrating consistent constraint satisfaction. For both of the problems \textbf{P1} and \textbf{P2}, a symmetric second-order discretization was applied at the interior points of the spatial domain, while near the boundaries, a non-symmetric second-order discretization was used to discretize spatial derivatives. 
\blue{We also employed the well-known scenario-based robust optimization technique \cite{ref:MC_SG-18} for a fair comparison (for benchmarking we only report our findings for \textbf{P2}). The scenario approach randomly sample a finite sequence of uncertainty instances and solve the ensuing robust optimization problems with constraints only on those instances. Such a technique is rooted in randomization and provides probabilistic guarantees on the continuous-time constraint satisfaction, and \emph{does not} provide any theoretical guarantee of satisfying \emph{the uncountable family of constraints} like we do. Indeed, in Fig. \ref{fig:glbopt_vs_scen} (left-hand subfig.) it can be observed that the control trajectory violates the maximum upper limit and thus there is no guarantee of robust constraint satisfaction.}

\begin{table}[h!]
\begin{tblr}{l c c c}
\hline[2pt]
Solver & CPU time (s) & Opt. Val. & Con. Sat.  \\ 
\hline[2pt]
\SetRow{green9}
\(\glbopt\): Diff. evol. &  \(242.88\) & \(5.27 \times 10^{-4}\)   & \cmark \\
\SetRow{green9}
\SetRow{green9}
\hspace{9mm} SA (iter=50) &  \(121.42\) & \(2.18 \times 10^{-3}\)    & \cmark \\
\SetRow{green9}
\hspace{9mm} \(\gradol\) &  \(19.2\) & \( 6.37 \times 10^{-3}\)    & \cmark \\

 \hline
Scenario: 21 samples &  \(2.34\) & \(3.24 \times 10^{-4}\)    & \xmark \\
\hspace{9mm} 100 samples &  \(11.12\) & \(5.57 \times 10^{-4}\)    & \xmark \\
\hline[2pt]
\end{tblr}
\centering
\vspace{2.5mm}
\caption{\blue{Numerical statistics of several algorithms in \(\glbopt\) (differential evolution, simulated annealing, and \(\gradol\)) and the scenario robust optimization techniques. `Opt. Val' and `Con Sat.' stands for optimal value and continuous-time constraint satisfaction, respectively. For each solver, the reported ``CPU time (in secs)" corresponds to the average over 200 independent runs, while the reported Opt. Val. is the maximum objective value attained across these 200 runs. For all solvers, every minimization problem is solved using \texttt{IPOPT} in order to ensure uniformity across implementations.}}
\label{tab:solver_time_comparison_r3}
\end{table}

\begin{figure}[h]
    \centering
    \begin{subfigure}[b]{0.48\linewidth}
        \centering
        \includegraphics[width=\linewidth,height=5cm]{results/D2D_Con_N50.pdf}
    \end{subfigure}
    \hfill
    \begin{subfigure}[b]{0.48\linewidth}
        \centering
        \includegraphics[width=\linewidth,height=4.6cm]{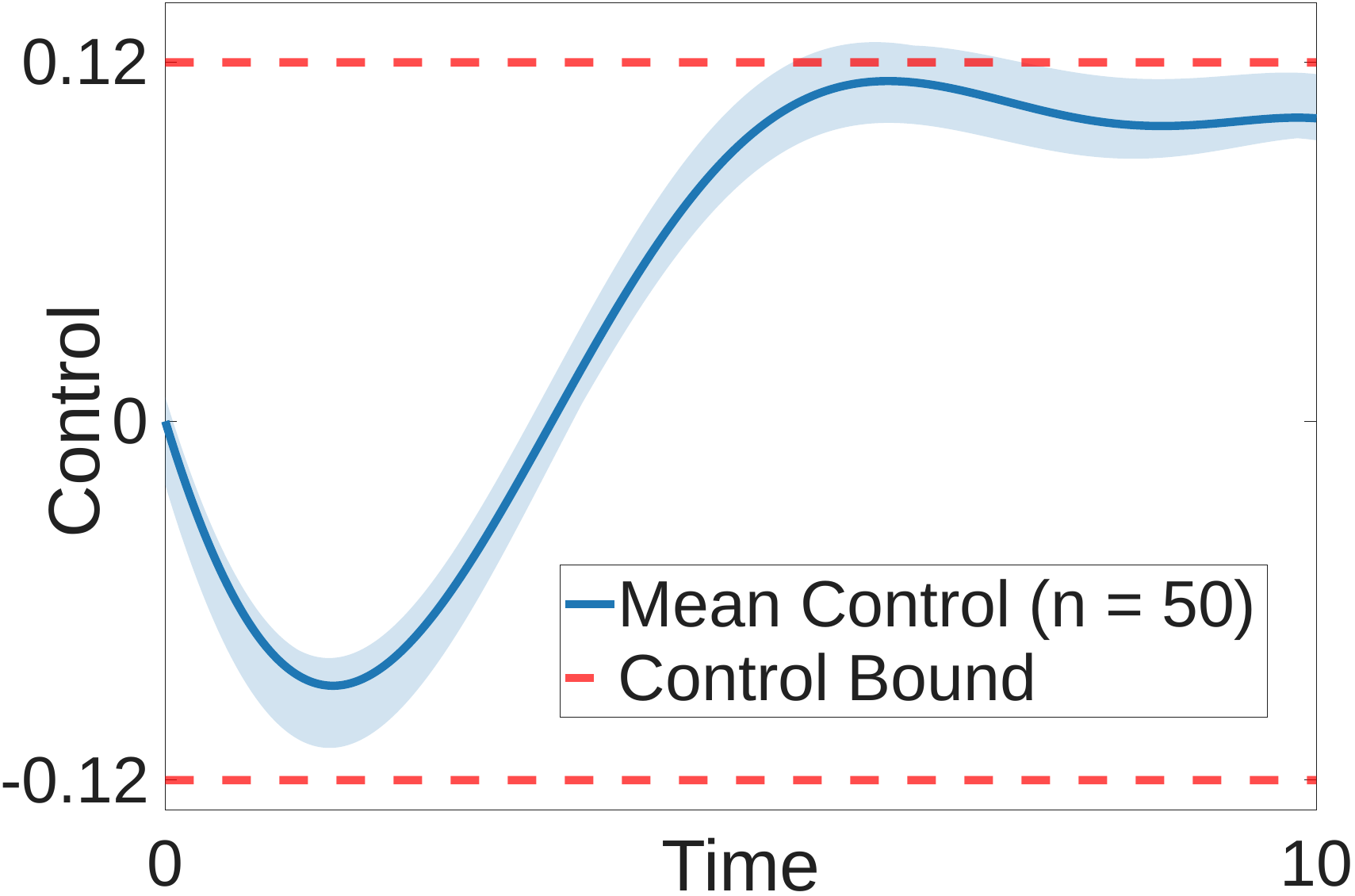}
    \end{subfigure}
    \caption{\blue{Both subfigures (left-hand - \(\glbopt\) with \(\gradol\) and right-hand - scenario) show the evolution of the boundary control input \(t \mapsto v(t) + w(t)\) for discretization level \(n = 50\), for Problem \textbf{P2} in Example \ref{numexp:heat_eq:example}.}}
    \label{fig:glbopt_vs_scen}
\end{figure}


\subsection{Reaction-Advection-Diffusion (R-A-D)}\label{numexp:RAD:example}
\blue{We consider the following R-A-D system \cite{ref:cruz2020boundary}} as another representative example to show that our algorithm works for a more general class of parabolic PDEs as discussed in Remark \ref{rem:on:general:systems:semi-disc}. The PDE is characterized by the following dynamics and accompanied by Robin boundary conditions
\begin{equation}
    \label{eq:rad_robin}
    \stpde_t(\spacepde,t) = \stpde_{\spacepde\spacepde}(\spacepde,t) + b\stpde_{\spacepde}(\spacepde,t) + \lambda \stpde(\spacepde,t),
\end{equation}
for all \((\spacepde,t) \in \Gamma\), where \(b = 1\) and \( \lambda = 0.3\), and as before \(\Gamma = \lcrc{0}{T} \times \loro{0}{1}\); \(\horizon>0\), with the following boundary conditions:
\begin{equation}
    \label{eq:robin_boundary}
     \stpde_{\spacepde}(0, t) + \frac{b}{2} \stpde(0, t) = 0, \,\,\,  \stpde_{\spacepde}(1, t) = \contpde(t) + \uncertpde(t)
\end{equation}
for all \(t \in \lcrc{0}{\horizon}\). We semi-discretize the system \eqref{eq:rad_robin}--\eqref{eq:robin_boundary} using the scheme established in \cite[Equation 25]{ref:SC:VN:motion:TAC25}. 
The initial and desired terminal states, for all \(x \in \loro{0}{1}\) are
\begin{equation}\label{eq:num:init_final_conditions_robin}
   x \mapsto y(x,0)\Let \bar{y}(\spacepde) \quad\text{and} \quad x \mapsto y(x,\horizon)\Let \terminalstpde(\spacepde) \nn
\end{equation}
which will be specified later. With the above ingredients, consider the following OCP
\begin{equation}
	\label{eq:heat_ocp_2}
\begin{aligned}
& \inf_{\contpde(\cdot) }	\sup_{\uncertpde(\cdot) }&&  \frac{1}{2} \int_{0}^{1} |\stpde(x,T) - \terminalstpde(x)|^2 \odif{x} \nn \\
& && + \frac{\lambda}{2}\int_{\tinit}^{\horizon} \contpde(t)^{2} \odif{t} - \frac{\gamma}{2} \int_{0}^{\horizon} \uncertpde(t)^2 \odif{t} \nn\\
&  \sbjto		&&  \begin{cases}
    \text{dynamics }\eqref{eq:rad_robin},\, \stpde(\spacepde, \tinit)= \bar{\stpde}(\spacepde),\,\eqref{eq:robin_boundary},\\ 
    -20 \le {\contpde(t) + \uncertpde(t)} \le 5,\,\, \text{for all\ } t \in \lcrc{0}{\horizon},\\
    \text{for all } \abs{\uncertpde(t)} \le \mu_\uncertpde  \text{ for all } t \in \lcrc{0}{\horizon}.
\end{cases}
\end{aligned}
\end{equation}
We consider two cases for the PDE system \eqref{eq:rad_robin} subject to the boundary condition \eqref{eq:robin_boundary}. In the first case, we set \( \contpde(t) = 0 \) for all \( t \in \lcrc{0}{20} \) to observe the inherent behavior of the system in the absence of control. In the second case, we address the OCP \eqref{eq:heat_ocp_2}, which represents the null control objective, i.e., steering a given initial state \(x\mapsto \bar{y}(x)\) to the null state \(x\mapsto y_{\Omega}(x)=0\). As before, for this problem, we solve the ensuing SIP employing the \(\glbopt\) architecture with the differential evolution routine with the same data selected and basis functions for the Example \ref{numexp:heat_eq:example}. The initial and final states and the problem parameters were fixed as
\begin{align}\label{num:D2D:SS:init:targ_2}
\loro{0}{1} \ni x \mapsto
\begin{cases}
\bar{y}(x) \Let 5 \Bigl(1-2\sin\bigl(\frac{3\pi x}{2}\bigr)\Bigr), \\
\terminalstpde(x) \Let 0,
\end{cases}
\end{align}
\((\horizon,\lambda, \gamma, \mu_\uncertpde) \Let (0.5, 10^{-5}, 10^{-3}, 0.01)\).

Note that the admissible uncertainties are assumed to be small relative to the control input, which is attributed to the unstable nature of the system with respect to the given boundary input, as illustrated by the uncontrolled response shown in Fig. \ref{fig:state:dist:Ex2:D2N} (middle subfigure), making it more difficult to control under uncertainty. Nevertheless, the control synthesized by the \(\glbopt\) architecture successfully steers the specified state in \eqref{num:D2D:SS:init:targ_2}, despite the presence of uncertainties and constraints. This behavior is illustrated in Fig. \ref{fig:state:dist:Ex2:D2N} (left-hand subfigure), with the corresponding control trajectory with \(n=50\) discretization is shown in Fig. \ref{fig:state:dist:Ex2:D2N} (right-hand subfigure).

\section{Acknowledgment}

The authors gratefully acknowledge Adi Ditkowski of Tel Aviv University for insightful suggestions on semi-discretization of parabolic PDEs shared via email correspondence.

\section{Conclusion}
\label{sec:conclusion}
 \blue{In this paper, we established an algorithmic architecture \(\glbopt\) for numerically solving constrained minmax optimal control of parabolic PDEs, based on semi-discretization of the PDE and an exact convex semi-infinite programming reformulation. Our technique leverages novel exactness guarantees from the convex SIP theory while directly enforcing uncountably many constraints in the continuous-time regime. Numerical illustrations along with comparison with existing robust optimization-driven techniques were provided. Natural future directions include theoretical investigation of the convergence of the discretized solutions and making the numerical architecture computationally faster.}

\bibliographystyle{ieeetr}
\bibliography{refs}

\appendices
\section{Auxiliary lemmas and proof of our main results}\label{appen:thrm:proofs}
This section presents the proofs of Theorem \ref{thrm:err_estimates} and Proposition \ref{prop:comp:conv:adparam:sets}, along with several supporting lemmas (and their proofs) that are essential for establishing Theorem \ref{thrm:value_func_equality}.

\begin{proof}[Proof of Proposition \ref{prop:comp:conv:adparam:sets}]
Recall that, the Minkowski difference of two sets \(S_1\) and \(S_2\) is \(S_1 \ominus S_2 \Let \aset[]{v - w \suchthat v \in S_1 \text{ and } w \in S_2}\). Consider the following sets
\begin{align}
\adparamsup \Let \aset[\big]{ \Param \in \Rbb^{ N_{\contpde}}
\suchthat \inprod{\Param}{\Reg(t)} \in \admcontpde \ominus \uncertsetode \text{ for all } t \in \lcrc{0}{\horizon}}, \nn
\end{align}
and
\begin{align}
    \adparamDsup \Let \aset[\big]{ \Paramw \in \Rbb^{N_{\uncertpde}}
    \suchthat \inprod{\Paramw}{\RegD(t)} \in \uncertsetode \text{ for all } t \in \lcrc{0}{\horizon}}.\nn
\end{align}
We establish that the sets \(\adparamsup\) and \(\adparamDsup\) are both compact and convex. Observe that the set \(\admcontpde \ominus \uncertsetode\) is a compact interval in \(\mathbb{R}\), since both \(\admcontpde\) and \(\uncertsetode\) are compact intervals in \(\mathbb{R}\).
For brevity, we provide a detailed proof only for \(\adparamDsup\), since similar to \(\uncertsetode\) being a compact interval in \(\mathbb{R}\), the set \(\admcontpde \ominus \uncertsetode\) is also a compact interval in \(\mathbb{R}\).

Note that the map \(\Paramw \mapsto \inprod{\Paramw}{\RegD(t)}\) is affine for each \(t \in \lcrc{0}{\horizon}\). Given that \(\uncertsetode\) is a compact interval in \(\mathbb{R}\), it is convex. Consequently, the set \(\adparamDsup\) is convex.  

For compactness, we show that \(\adparamDsup\) is closed and bounded. To this end, fixing \(t\in \lcrc{0}{\horizon}\), define the set
\(
    \adparamDsup^t \Let \aset[]{\Paramw \in \Rbb^{N_{\uncertpde}} \suchthat a_1 \le \inprod{\Paramw}{\RegD(t)} \le a_2}. \nn
\)
Observe that \(\Paramw \mapsto \inprod{\Paramw}{\RegD(t)}\) is linear in \(\Paramw\) and thus continuous and \(\RegD(t) \in \Rbb^{N_{\uncertpde}}\) is fixed. Then \(\adparamDsup^t\) is the intersection of the closed sets \(\aset[]{\Paramw \suchthat \inprod{\Paramw}{\RegD(t)} \ge a_1}\) and \(\aset[]{\Paramw \suchthat \inprod{\Paramw}{\RegD(t)} \le a_2}\) and thus closed. Since the preimage of a closed set under a continuous mapping is closed, so is \(\adparamDsup^t\), and \(
    \adparamDsup = \bigcap_{t \in \lcrc{0}{\horizon}} \adparamDsup^t \nn
\)
is closed since it is an intersection of closed sets. We show that \(\adparamDsup\) is also bounded. Towards that end, define the \emph{Gram matrix}
\begin{align}
 \gram \Let \int_{0}^{\horizon} \RegD(t) \RegD(t)^{\top}\odif{t} \in \Rbb^{N_{\uncertpde} \times N_{\uncertpde}}, \nn
\end{align}
which is symmetric and positive definite. Indeed, for any \(\beta \in \Rbb^{N_{\uncertpde}}\) we have 
\(\beta^{\top}\gram\beta = \int_0^{\horizon}\inprod{\beta}{\RegD(t)}^2 \odif{t}\), which implies that \(\gram\succeq 0\); but the only \(\beta\) with \(\beta^{\top}\gram\beta = 0\) is \(\beta =0\), which follows readily from the linear independence of the functions \(\dicD(\cdot)\); see Definition \ref{defn:discrete_admcon}. Thus \(\gram \succ 0\) and consequently all eigenvalues of \(\gram\) are positive; we let the smallest such eigen value be \(\lambda_{\text{min}}(\gram) >0\). We have the following immediate bound: \(\beta^{\top}\gram\beta \ge \lambda_{\text{min}}(\gram)\norm{\beta}_2^2\). Observing the fact that \(\lcrc{0}{\horizon} \ni t \mapsto \inprod{\beta}{\RegD(t)}\) is continuous and bounded, we have
\begin{align}\label{eq:bound:1}
 \lambda_{\text{min}}(\gram)\norm{\beta}_2^2 \le    \beta^{\top}\gram\beta &= \int_0^{\horizon} \inprod{\beta}{\RegD(t)}^2 \odif{t} \nn \\ 
 &\le \horizon \sup_{t \in \lcrc{0}{\horizon}} \abs{\inprod{\beta}{\RegD(t)}}^2.
\end{align}
But \(\inprod{\beta}{\RegD(t)} \in \lcrc{a_1}{a_2}\) and thus the right-hand side of the inequality \eqref{eq:bound:1} admits the upper bound \(\sup_{t \in \lcrc{0}{\horizon}} \abs{\inprod{\beta}{\RegD(t)}} \le M_0 \Let \max\{\abs{a_1},\abs{a_2}\}\). Consequently \(\norm{\beta}_2 \le \sqrt{\frac{\horizon}{\lambda_{\text{min}}(\gram)}} M_0\) and the set \(\adparamDsup\) is bounded. Compactness now follows from the Heine-Borel theorem \cite[Theorem 2.41]{ref:Rud-Analysis}. 


Observe that the map \(\Paramw \mapsto \inprod{\Paramw}{\RegD(0)}\) is continuous and affine. Therefore, the set 
\(
\left\{ \Paramw \in \mathbb{R}^{\uncertpde} \,\middle|\, \inprod{\Paramw}{\RegD(0)} = \parampde(1) - \contpde_0 \right\}
\)
is closed and convex. \blue{Moreover, since \(\aset[]{\dicD_i(\cdot)\suchthat i \in \Nz} \subset  \mathcal{C}^{3,\holder}(\lcrc{0}{\horizon};\Rbb)\), \(\Paramw \mapsto \norm{\inprod{\Paramw}{\RegD(\cdot)}}_{\mathcal{R}}\) is continuous and convex. Consequently, the set \(M_{\Paramw} \Let \aset[]{\Paramw \suchthat \norm{\inprod{\Paramw}{\RegD(\cdot)}}_{\mathcal{R}} \le \adboundw}\) is closed and convex.} It follows that the set
\begin{equation}
\blue{\adparamD = \adparamDsup \cap \left\{ \Paramw \in \mathbb{R}^{\uncertpde} \,\middle|\, \inprod{\Paramw}{\RegD(0)} = \parampde(1) - \contpde_0 \right\} \cap M_{\Paramw}}, \nn    
\end{equation}
is compact and convex. Since both sets involved in the intersection are closed and convex and \(\adparamDsup\) is compact. The intersection of a compact set with a closed set is again compact.

Following the same line of arguments as above, \(\adparamsup\) is compact, and by definition, we have \( \adparam \subset \adparamsup \). To establish the compactness of \( \adparam \), it therefore suffices to show that \( \adparam \) is closed. To this end, fixing a \((t,\beta)\) we define
\begin{equation}
    \adparam^{(\Paramw, t)} \Let \aset[]{\Param \in \Rbb^{\contpde} \suchthat \inprod{\Param}{\Reg(t)} + \inprod{\Paramw}{\RegD(t)} \in \admcontpde}.\nn
\end{equation}
Since \( \admcontpde \) is a compact interval in \( \Rbb \), say \(\lcrc{\hat{a}_1}{\hat{a}_2}\), we can rewrite this set as \(\adparam^{(\Paramw, t)} = \bigl\{\Param \in \Rbb^{\contpde} \vert \inprod{\Param}{\Reg(t)} \in \lcrc{\hat{a}_1 - \inprod{\Paramw}{\RegD(t)}}{\hat{a}_2 - \inprod{\Paramw}{\RegD(t)}}\bigr\}\), which is closed and convex. Similarly, the set \(M_v \Let \aset[]{\Param \in \Rbb^{\contpde} \suchthat \inprod{\Param}{\Reg(0)} = \contpde_0}\) is also closed and convex. \blue{Moreover, since \(\aset[]{\dicC(\cdot)\suchthat i \in \Nz} \subset  \mathcal{C}^{3,\holder}(\lcrc{0}{\horizon};\Rbb)\), \(\Param \mapsto \norm{\inprod{\Param}{\Reg(\cdot)}}_{\mathcal{R}}\) is continuous and convex. Consequently, the set \(M_{\Param} \Let \aset[]{\Param \suchthat \norm{\inprod{\Paramw}{\RegD(\cdot)}}_{\mathcal{R}} \le \adboundv}\) is closed and convex.} Thus, the intersection
\begin{equation}
    \blue{\adparam = \biggl(\bigcap_{\Paramw \in \adparamD} \bigcap_{t \in \lcrc{0}{\horizon}} \adparam^{(\Paramw, t)}\biggr) \cap M_v \cap M_{\Param}} \nn
\end{equation}
is closed and convex. Since \(\adparamsup\) is compact, compactness of \(\adparam\) follows. 
\end{proof}
The following lemmas establish convexity properties of certain mappings and sets, which will be employed in the proof of Theorem~\ref{thrm:value_func_equality}.
\begin{lemm}\label{lemm:convexity:objective}
Recall the expression of the trajectory \(\lcrc{0}{\horizon} \ni s \mapsto \st(s;\param,\Param,\Paramw) \in \Rbb^n\) in \eqref{eq:p_sol} and the objective function \eqref{eq:parametrized_cost}
\begin{align}
       \objectivedisc\bigl(\param, \inprod{\Param}{\Reg(\cdot)},\inprod{\Paramw}{\RegD(\cdot)} \bigr) & = 
       \frac{\lambda}{2}\int_{\tinit}^{\horizon} \inprod{\Param}{\Reg(\tau)}^{2} \odif{\tau}  \nn\\
       & \hspace{-30mm} + \frac{h}{2} \biggl{\|} e^{A_n\horizon}\stparam(\tinit) + \int_{0}^{\horizon}e^{A_n(\horizon-\tau)}B_n
         \bigl(\inprod{\Param}{\Reg(\tau)}\nn \\ & \hspace{-30mm}+ \inprod{\Paramw}{\RegD(\tau)}\bigr)\odif{\tau} 
        - \st_{\spacedomainpde}\biggr{\|}^2        - \frac{\gamma}{2}\int_{\tinit}^{\horizon} \inprod{\Paramw}{\RegD(\tau)}^2  \odif{\tau} \nn
\end{align}
The mapping \((\slack,\Param) \mapsto  \objectivedisc\bigl(\param, \inprod{\Param}{\Reg(\cdot)},\inprod{\Paramw}{\RegD(\cdot)} \bigr) - \slack\) is convex. 
\end{lemm}

\begin{proof}
Fix \((\Paramw, \RegD(\cdot)) \in \adparamD \times \dict_{\uncertpde}\) and \(\param \in \Rbb^n\). We define the map 
\begin{align}
    (\slack,\Param) \mapsto \objmap(\slack,\Param) & \Let  \objectivedisc\bigl(\param, \inprod{\Param}{\Reg(\cdot)},\inprod{\Paramw}{\RegD(\cdot)} \bigr) - \slack. \nn
\end{align}
Then \(\objmap(\cdot)\) is 
\begin{align}
       \objectivedisc\bigl(\param, \inprod{\Param}{\Reg(\cdot)},\inprod{\Paramw}{\RegD(\cdot)} \bigr) - \slack & = 
        \frac{\lambda}{2}\int_{\tinit}^{\horizon} \inprod{\Param}{\Reg(t)}^{2} \odif{t} \nn\\& 
        \hspace{-35mm} + \frac{h}{2} \bigg{\|} e^{A_n\horizon}\stparam(\tinit) + \int_{0}^{\horizon}e^{A_n(\horizon-\tau)}B_n
        \bigl(\inprod{\Param}{\Reg(\tau)}\nn \\ & \hspace{-35mm}+ \inprod{\Paramw}{\RegD(\tau)}\bigr)\odif{\tau} 
       - \st_{\spacedomainpde}\bigg{\|}^2 - \frac{\gamma}{2}\int_{\tinit}^{\horizon} \inprod{\Paramw}{\RegD(t)}^2  \odif{t} - \slack.\nn
\end{align}
The derivative of \(\objmap(\cdot)\) acting on the tuple \((d_{\slack},d_{\Param}) \in \Rbb \times \Rbb^{N_{\contpde}}\) is given by
\begin{align}
 &   \frac{\partial}{\partial(\slack,\Param)}\objmap(\slack,\Param) \cdot (d_{\slack},d_{\Param}) = h\Bigg{\langle} \int_{0}^{\horizon}e^{A_n(\horizon-\tau)}
         B_n\inprod{d_{\Param}}{\Reg(\tau)}\odif{\tau}, \nn\\ 
& \hspace{2mm} \biggl(e^{A_n\horizon}\stparam(\tinit) + \int_{0}^{\horizon}e^{A_n(\horizon-\tau)}B_n
         \bigl(\inprod{\Param}{\Reg(\tau)}\nn + \inprod{\Paramw}{\RegD(\tau)}\bigr)\odif{\tau} 
        \\ & \hspace{2mm}-\st_{\spacedomainpde}\biggr)\Bigg{\rangle} 
        + \lambda\int_{\tinit}^{\horizon} \inprod{d_{\Param}}{\Reg(t)}\inprod{\Param}{\Reg(t)} \odif{t} - d_{\slack}. \nn
\end{align}
Similarly the Hessian of \(\objmap(\cdot)\), acting on the tuple  \(\bigl((d_{\slack},d_{\Param}),(d_{\slack},d_{\Param})\bigr) \in \bigl(\Rbb \times \Rbb^{N_{\contpde}}\bigr)^2\) is given by
\begin{align*}
\frac{\partial^2}{\partial^2(\slack,\Param)}\objmap(\slack,\Param) 
&\cdot (d_{\slack},d_{\Param}) \cdot (d_{\slack},d_{\Param}) \\
&= h\Bigg\langle \int_{0}^{\horizon}e^{A_n(\horizon-\tau)} B_n
       \inprod{d_{\Param}}{\Reg(\tau)}\,d\tau, \\
&\quad \int_{0}^{\horizon}e^{A_n(\horizon-\tau)} B_n
       \inprod{d_{\Param}}{\Reg(\tau)}\,d\tau \Bigg\rangle \\
&\quad + \lambda \int_{\tinit}^{\horizon} \inprod{d_{\Param}}{\Reg(t)}^2 dt \;\;\geq 0.
\end{align*}
In view of the smoothness of functions involved in the computations, we applied the derivative operator within the integrals, as justified by \cite[Chapter XVII, Theorem 8.2]{ref:lang2013undergraduate}. Thus \(\objmap(\cdot)\) is twice continuously differentiable and convex in \((\slack,\Param)\). 
\end{proof}

\begin{lemm}
    \label{lem:Aux_lem}
Consider the OCP \eqref{eq:og_OCP} and its associated convex SIP \eqref{eq:ocp_semi-discrete_parametrized}, along with their respective data and notations and suppose that Assumption \ref{assum:slater_cond} is in force. Define the feasible set 
    \begin{align*}
        \feasSip \Let \left\{ (\slack, \Param) \in \Rbb \times \Rbb^{N_{\contpde}} \mid \text{the constraints in \eqref{eq:ocp_semi-discrete_parametrized} holds} \right\}.
    \end{align*}
    Then \(\feasSip\) is non-empty, closed and convex.
\end{lemm}

\begin{proof}
Nonemptiness of \(\feasSip\) follows directly from Assumption \ref{assum:slater_cond}. For a fixed \((t, \Paramw) \in \lcrc{0}{\horizon} \times \adparamD\) and \(\bar{z} \in \Rbb^n\), define the set
\begin{align}\label{eq:feas:set:t:eta}
    \feasSip^{t, \Paramw} \Let \left\{(\slack, \Param) \;\middle\vert\;  
    \begin{array}{@{}l@{}}
        \objectivedisc\bigl(\param, \inprod{\Param}{\Reg(\cdot)}, \inprod{\Paramw}{\RegD(\cdot)})\bigr) - \slack \leqslant 0,\\
        \inprod{\Param}{\Reg(0)} = \contpde_0,\,\, \stparam(0)=\param, \\
        \text{and } \inprod{\Param}{\Reg(t)} + \inprod{\Paramw}{\RegD(t)} \in \admcontpde\\
        \end{array}
        \right\}.    
 \end{align}
Note that \(\admcontpde\) is compact and convex from the problem data \ref{eq:prob:data:3}--\ref{eq:prob:data:4} and the mapping \((s,\Param) \mapsto \inprod{\Param}{\Reg(t)} + \inprod{\Paramw}{\RegD(t)}\) is continuous and linear (thus convex). Also, the mapping
\((s,\Param) \mapsto  \objectivedisc\bigl(\param, \inprod{\Param}{\Reg(\cdot)},\inprod{\Paramw}{\RegD(\cdot)} \bigr)-s\)
is continuous and convex as shown in Lemma \ref{lemm:convexity:objective}. 
Thus, \(\feasSip^{t,\Paramw}\) is closed and convex for every \((t, \Paramw) \in \lcrc{0}{\horizon} \times \adparamD\). Then the feasible set 
\(\feasSip \Let  \bigcap_{(t, \Paramw) \in \totaluncertset} \feasSip^{t, \Paramw}\)
is also closed since it is an intersection of closed and convex sets.
\end{proof}

\begin{proof}[Proof of Theorem \ref{thrm:value_func_equality}]
Let us commence by examining the statement \ref{thrm:value_func_equality_0}, which addresses Lipschitz continuity of the mapping \((\tseq,\bseq) \to \gfunc(\tseq,\bseq;\param)\) specified in equation \eqref{eq:g_func}, for a fixed \(\param\).

Fix \(\param\) from \(\fsblset\) and a point \((\widehat{\slack}, \widehat{\Param}, \widehat{\tseq}, \widehat{\bseq})\) belonging to the domain \(\mathbb{R} \times \adparam \times [0,\horizon]^{\dvar} \times \adparamD^{\dvar}\). Notice that
	\begin{enumerate}[label=\textbf{(\(\gfunc\)-\roman*)}, leftmargin=*, widest=ii]
		\item \label{still:hypo:1} the objective function \((\slack,\Param,\tseq,\bseq) \mapsto r\) in \eqref{eq:g_func} is continuous around \((\widehat{r}, \widehat{\Param}, \widehat{\tseq}, \widehat{\bseq})\) and is convex in \((\slack,\Param)\) for every fixed \((\tseq,\bseq)\).
	\end{enumerate}
The objective function \(\objectivedisc(\param, \cdot, \cdot)\) is continuous due to Lemma \ref{lemm:convexity:objective} and since the product set \(\adparam \times \adparamD\) is compact from Proposition \ref{prop:comp:conv:adparam:sets}. Thus, there exists a constant \(S \Let S(\param) > 0\) such that \(\objectivedisc(\param, \cdot, \cdot) \le S(\param)/2\). Given that \(\objectivedisc(\param,\cdot,\cdot)\) yields non-negative values, it is sufficient to confine the decision domain of \eqref{eq:g_func} to the compact region \(\lcrc{0}{ S(\param)} \times \adparam\) rather than \(\Rbb \times \adparam\).
\begin{enumerate}[label=\textbf{(\(\gfunc\)-\roman*)}, leftmargin=*, widest=ii, start=2]

\item \label{still:hypo:2} For each \(i\in \aset[]{1,\ldots,\dvar}\), the mapping \((\slack,\Param,\tseq,\bseq)\mapsto \inprod{\Param}{\Reg(0)} - \contpde_0\) in \eqref{eq:g_func} is continuous around \((\widehat{\slack}, \widehat{\Param}, \widehat{\tseq}, \widehat{\bseq})\). Also, for each fixed \(\bigl(\widehat{\tseq},\widehat{\bseq}\bigr)\), from Lemma \ref{lemm:convexity:objective}, it follows that \((\slack,\Param)\mapsto \objectivedisc\bigl(\param,\inprod{\Param}{\Reg(\cdot)},\inprod{\Paramw^i}{\RegD(\cdot)} \bigr) - \slack\) is convex.

\item \label{still:hypo:3} For each \(i\in \aset[]{1,\ldots,\dvar}\), the mapping
\[(\slack,\Param,\tseq,\bseq)\mapsto \objectivedisc\bigl(\param,\inprod{\Param}{\Reg(\cdot)},\inprod{\Paramw^i}{\RegD(\cdot)} \bigr) - \slack\] 
in \eqref{eq:g_func} is affine and therefore continuous and convex around \((\widehat{\slack}, \widehat{\Param}, \widehat{\tseq}, \widehat{\bseq})\).

        \item \label{still:hypo:4} For every fixed \((\tseq,\bseq)\) and \(i \in \aset[]{1,\ldots,\dvar}\), the control constraint set
        \begin{align}\label{eq:control_feas_set}
        \left\{(\slack,\Param)\in \Rbb \times \adparam \;\middle\vert\; \begin{array}{@{}l@{}}
        \inprod{\Param}{\Reg(t^i)} + \inprod{\Paramw^i}{\RegD(t^i)}\in \admcontpde
        \end{array}
        \right\} 
        \end{align}
        is convex due to Lemma \ref{lem:Aux_lem}.
	\end{enumerate}
    \vspace{1mm}
	Assumption \ref{assum:slater_cond} ensures the existence of \((\slack', \Param')\in\lcrc{0}{S(\param)}\times\adparam\) such that every constraint is strictly satisfied at \((\slack', \Param', \widehat{\tseq}, \widehat{\bseq})\). Moreover, the control constraint \(\inprod{\Param}{\Reg(t)} \in \admcontpde\) is smooth due to the nature of the dictionary \(\dict_{\contpde}\).

As a result, the optimization problem \eqref{eq:g_func} constitutes a convex program featuring continuously differentiable data \cite[Chapter 7]{ref:param_opt_still}. Invoking \cite[Exercise 7.1, p. 102]{ref:param_opt_still} alongside Assumption \ref{assum:slater_cond}, for a fixed \((\ol{\tseq},\ol{\bseq})\) in \(\totaluncertset\), the Mangasarian-Fromovitz Constraint Qualification (MFCQ) conditions outlined in \cite[Theorem 4.2]{ref:FiaIsh-90} are upheld for all \((\slack,\Param)\) within the feasible subset of \([0, S(\param)] \times \adparam\), a compact domain. Consequently, the function \(\gfunc(\cdot, \cdot; \xz)\) is locally Lipschitz continuous in the vicinity of \((\ol{\tseq},\ol{\bseq})\) within \(\totaluncertset\), as affirmed by \cite[Theorem 4.2]{ref:FiaIsh-90}. Since \(\totaluncertset = \lcrc{0}{\horizon}^{\dvar} \times \adparamD^{\dvar}\) is compact, \cite[Theorem 2.1.6]{ref:CobMicNic-19} assures that the function \(\gfunc(\cdot, \cdot; \param)\) is globally Lipschitz continuos on \(\totaluncertset\). This establishes the proof of the statement \ref{thrm:value_func_equality_0}.

The existence of an optimizer \((\tseq\as(\param), \bseq\as(\param)) \in \totaluncertset\) is a direct consequence of Weierstrass’s extreme value theorem \cite[Box 1.3, pp. 9--10]{ref:santambrogio2023course}. Indeed, \(\gfunc(\cdot, \cdot; \param)\) is globally Lipschitz continuous on the compact domain \(\totaluncertset\) and thus existence of an optimizer \((\tseq\as(\param), \bseq\as(\param)) \in \totaluncertset\) follows immediately, proving \ref{thrm:value_func_equality_1}.

We now prove the final assertion \ref{thrm:value_func_equality_2}. Recall the feasible set \(\feasSip\) and the set \(\feasSip^{t,\Paramw}\) defined in Lemma \ref{lem:Aux_lem} and \eqref{eq:feas:set:t:eta}. Note that the existence of \(\bigl(t^i,\uncertparam^i(\cdot)\bigr)_{i=1}^{\dvar}\) directly implies the existence of \(\bigl(\tseq,\bseq\bigr) \in \totaluncertset\) such that the strict feasibility condition in Assumption \ref{assum:slater_cond} corresponding to the relaxed problem \eqref{eq:g_func} holds, i.e., there exists a point \((\slack, \Param) \in \Rbb \times \adparam\) such that for every \(\bigl(\tseq,\bseq\bigr) \in \totaluncertset\), the interior of the set 
 \[
 \bigcap_{\mathclap{\substack{\bigl((t^i)_{i=1}^{\dvar},(\Paramw^i)_{i=1}^{\dvar}\bigr) \in \totaluncertset}}} \feasSip^{t_i,\Paramw^i}
 \]
 is nonempty. Also recall that Lemma \ref{lem:Aux_lem} ensures that the set of feasible \((\slack,\Param)\)
  \begin{align*}
        \feasSip \Let \left\{ (\slack, \Param) \in \Rbb \times \Rbb^{N_{\contpde}} \mid \text{constraints in \eqref{eq:ocp_semi-discrete_parametrized} hold} \right\}
    \end{align*}
is closed and convex. Note that in \eqref{eq:ocp_semi-discrete_parametrized}
\begin{itemize}[leftmargin=*]
\item the objective function \((\slack,\Param) \mapsto \slack\) convex and continuous. For a fixed \((t,\Paramw) \in \lcrc{0}{\horizon} \times \adparamD\) and a given \(\param \in \fsblset\) the constraint map \((\slack,\Param) \mapsto \objectivedisc\bigl(\param,\inprod{\Param}{\Reg(\cdot)},\inprod{\Paramw}{\RegD(\cdot)}\bigr)-\slack\) is convex (see Lemma \ref{lemm:convexity:objective}). The set 
\[\aset[]{(\slack,\Param)\in \Rbb \times \Rbb^{N_{\contpde}} \suchthat \; \begin{array}{@{}l@{}}
        \Param \in \adparam
        \end{array}
        }\]
is convex (see Proposition \ref{prop:comp:conv:adparam:sets}).

\item Let \(\param \in \fsblset\) be given. Joint continuity of the constraint maps \[\ftuple \mapsto \objectivedisc\bigl(\param,\Param\Reg(\cdot),\Paramw\RegD(\cdot)\bigr)-\slack\] and \(\ftuple \mapsto \Param \Reg(t)\) follows from Lemma \ref{lemm:convexity:objective} and Proposition \ref{prop:comp:conv:adparam:sets}. 

\item The feasible set \(\feasSip\) is closed and convex (see Lemma \ref{lem:Aux_lem}) and the constraint index set \(\totaluncertset \Let \lcrc{0}{\horizon}\times \adparamD\) is compact. 
\end{itemize}
Thus, all the hypotheses of Theorem \cite[Theorem 1.1]{ref:DasAraCheCha-22} are verified. The assertion \ref{thrm:value_func_equality_1} assures existence of an optimizer \(\bigl(\tseq^{\ast}(\param),\bseq^{\ast}(\param)\bigr) \in \totaluncertset\) of the global maximization problem \eqref{e:global_max_prob}, and from \cite[Theorem 1.1]{ref:DasAraCheCha-22} along with the preceding arguments, it follows that for every \(\param\in \fsblset\) the equality
 \[
 \valuefunc(\param) = \gfunc(\tseq\as(\param),\bseq\as(\param); \param)
 \]
holds. The proof is complete.
\end{proof}

\end{document}